\documentclass[12pt]{amsart}
\usepackage{amsmath,amssymb,latexsym, amsthm, amscd, mathrsfs, stmaryrd} 

\usepackage[all]{xy}
\usepackage{hyperref}
\usepackage{color}
\hypersetup{colorlinks,linkcolor=blue,urlcolor=cyan,citecolor=blue}

\newtheorem{innercustomthm}{{\bf Theorem}}
\newenvironment{customthm}[1]
  {\renewcommand\theinnercustomthm{#1}\innercustomthm}
  {\endinnercustomthm}

\newcommand{\nc}{\newcommand}
\nc{\browntext}[1]{\textcolor{brown}{#1}}
\nc{\greentext}[1]{\textcolor{green}{#1}}
\nc{\redtext}[1]{\textcolor{red}{#1}}
\nc{\bluetext}[1]{\textcolor{blue}{#1}}
\nc{\brown}[1]{\browntext{ #1}}
\nc{\green}[1]{\greentext{ #1}}
\nc{\red}[1]{\redtext{ #1}}
\nc{\blue}[1]{\bluetext{ #1}}
\nc{\zb}[1]{\redtext{From zb: #1}}

\newcommand{\ff}{B}
\newcommand{\tK}{\widetilde{K}}

\newtheorem{thm}{Theorem}  [section]
\newtheorem{cor}[thm]{Corollary}
\newtheorem{lem}[thm]{Lemma}
\newtheorem{prop}[thm]{Proposition}
\newtheorem{conj}[thm]{Conjecture} 

\theoremstyle{remark}
\newtheorem{rem}[thm]{Remark}

\numberwithin{equation}{section}

\newcommand{\mbf}{\mathbf}

\newcommand{\mrm}{\mathrm}

\newcommand{\ev}{\bar{0}}
\newcommand{\odd}{\bar{1}}

\newcommand{\tU}{\widetilde{\mathbf{U}}}

\newcommand{\diag}{\mrm{diag}}

\newcommand{\Iblack}{\I_{\bullet}}
\newcommand{\wb}{w_\bullet}

\newcommand{\btau}{\tau}
\def \tf{{\widetilde{y}}}
\def \tk{{\widetilde{k}}}
\newcommand{\la}{\lambda}

\newcommand{\K}{\mathbb K}
\newcommand{\Iw}{\I_{\circ}}
\newcommand{\id}{\text{id}}

\newcommand{\Iwhite}{\I_{\circ}}

\newcommand{\bbZ}{\mathbb Z}
\newcommand{\one}{\mathbf 1}
\newcommand{\oldone}{\mathbf 1}
\newcommand{\onestar}{\one^\star}
\newcommand{\ov}{\overline}
\newcommand{\qbinom}[2]{\begin{bmatrix} #1\\#2 \end{bmatrix} }

\newcommand{\U}{\mbf U}
\newcommand{\Udot}{\dot{\mbf U}}

\newcommand{\Ui}{{\mbf U}^\imath}

\newcommand{\vs}{\varsigma}

\newcommand{\Z}{\mathbb Z}
\newcommand{\I}{\mathbb I}

\newcommand{\T}{\mbf T}

\nc{\fprime}{\bold{'f}}
\def \m{{m}}

\def \bU{{\mathbf U}}
\def \bvs{{\boldsymbol{\varsigma}}}

\newcommand{\tUi}{\widetilde{{\mathbf U}}^\imath}

\allowdisplaybreaks

\begin{document}

\title[Serre-Lusztig relations for  $\imath${}quantum groups]{Serre-Lusztig relations for  $\imath${}quantum groups}

\author[Xinhong Chen]{Xinhong Chen}
\address{Department of Mathematics, Southwest Jiaotong University, Chengdu 610031, P.R.China}
\email{chenxinhong@swjtu.edu.cn}

\author[Ming Lu]{Ming Lu}
\address{Department of Mathematics, Sichuan University, Chengdu 610064, P.R.China}
\email{luming@scu.edu.cn}

\author[Weiqiang Wang]{Weiqiang Wang}
\address{Department of Mathematics, University of Virginia, Charlottesville, VA 22904}
\email{ww9c@virginia.edu}



\begin{abstract}
Let $(\bf U, \bf U^\imath)$ be a quantum symmetric pair of Kac-Moody type. The $\imath$quantum groups $\bf U^\imath$ and the universal $\imath$quantum groups $\widetilde{\bf U}^\imath$ can be viewed as a generalization of quantum groups and Drinfeld doubles $\widetilde{\bf U}$. In this paper we formulate and establish Serre-Lusztig relations for $\imath$quantum groups in terms of $\imath$divided powers, which are an $\imath$-analog of Lusztig's higher order Serre relations for quantum groups. This has applications to braid group symmetries on $\imath$quantum groups.
\end{abstract}

\maketitle
\setcounter{tocdepth}{1}
\tableofcontents

\section{Introduction}

\subsection{Background}

For a Drinfeld-Jimbo quantum group $\U$ associated to a generalized Cartan matrix $(a_{ij})_{i,j\in \I}$, Lusztig \cite[Ch. 7]{Lu93} formulated higher order Serre relations, which we shall refer to as Serre-Lusztig relations in this paper. The Serre-Lusztig relations have rich connections with braid group actions and further applications to the finer algebraic structures of quantum groups, cf. \cite[Part~VI]{Lu93}.

Let $i\neq j\in \I$, $e =\pm 1$, $m\in \Z$ and $n\in \Z_{\ge 0}$. In terms of the elements
\begin{equation}  \label{eq:fij}
f_{i,j;n,m,e}^+ := \sum_{r+s=m} (-1)^r q_i^{er(1-na_{ij}-m)} E_i^{(r)} E_j^{(n)} E_i^{(s)},
\end{equation}
the Serre-Lusztig relations (cf. \cite[7.1.1, 7.15]{Lu93}) are expressed as
\begin{align}  \label{eq:LS}
f_{i,j;n,m,e}^+ =0,
\qquad \text{ for } m \geq 1 -na_{ij}.
\end{align}
The standard $q$-Serre relation is recovered at $n=1$ and $m=1-a_{ij}$. 

Serre-Lusztig relations have also played a crucial role in the XXZ spin chain with periodic boundary conditions and the superintegrable chiral Potts model, cf. \cite{KM01, ND08} and references therein.

For our purpose in this paper, it is helpful  to envision the Serre-Lusztig relations \eqref{eq:LS} in the following 2 steps.

\begin{enumerate}
\item[$\triangleright$]
One makes an Ansatz from the standard $q$-Serre relation to guess the formula for the {\em Serre-Lusztig relations of minimal degree}   (i.e., \eqref{eq:LS} for $m=1-na_{ij}$):
\begin{align}
  \label{eq:minSLQG}
 f_{i,j;n,1 -na_{ij},e}^+ = \sum_{r+s=1-na_{ij}} (-1)^r E_i^{(r)} E_j^{(n)} E_i^{(s)} =0.
\end{align}

\item[$\triangleright$]
 Starting with \eqref{eq:minSLQG}, the formula for $f_{i,j;n,m,e}^+$ and the relation \eqref{eq:LS} can be derived by the recursion formulas in \cite[Lemma~7.1.2]{Lu93} below: for $e =\pm 1, m \in \Z,$
\begin{align}
  \label{eq:Lrec}
q_i^{-e(na_{ij} +2m)} E_i f_{i,j; n,m,e}^+ - f_{i,j;n,m,e}^+ E_i
&= - [m+1]_{i} f_{i,j;n,m+1,e}^+,
\\
F_i f_{i,j;n,m,e}^+ - f_{i,j;n,m,e}^+ F_i
&= [na_{ij} +m-1]_{i} \tK_{-ei} f_{i,j;n,m-1,e}^+.
   \label{eq:LrecF}
\end{align}
\end{enumerate}

%
\vspace{2mm}

Let $\tU$ be the Drinfeld double of $\U$. Let $(\U, \Ui)$ be a quantum symmetric pair \cite{Le99, Ko14}, and let $(\tU, \tUi)$ be the universal quantum symmetric pair \cite{LW19a}. The so-called $\imath$quantum groups $\Ui$ and $\tUi$ can be viewed as generalizations of quantum groups, just as real Lie groups can be viewed as a generalization of complex Lie groups.
The definition of $\Ui$ and $\tUi$ is based on a Satake diagram or an admissible pair $(\Iblack, \tau)$: a partition $\I =\Iblack \sqcup \Iwhite$ with $\Iblack$ of finite type and (possibly trivial) Dynkin diagram involution $\tau$, which satisfy some compatibility conditions. The $\Ui =\Ui_{\bvs}$ depend on parameters $\bvs=(\vs_i)\in  (\K(q)^\times)^{\Iwhite}$, while $\tUi$ has additional Cartan subalgebra generators which produce various central elements (which controls the parameters $\bvs$). The $\imath$quantum groups $\tUi, \Ui$ are called quasi-split if $\Iblack =\emptyset$, and split if in addition $\tau =\id$.

The theory of canonical basis for quantum groups \cite{Lu90, Lu93} has been generalized to the setting of $\imath$quantum groups \cite{BW18b}. The rank 1 $\imath$canonical basis for $\Ui(\mathfrak{sl}_2)$ \cite{BW18a, BeW18} gives rise to the $\imath$divided powers $B^{(m)}_{i,\overline{p}}$ associated with $i\in \Iw$ with $\tau i =i = \wb i$ in $\Ui$ or $\tUi$, where $\wb$ denotes the longest element in the Weyl group $W_{\Iblack}$.  In contrast to the usual divided powers in quantum groups, the $\imath$divided powers, denoted by $B^{(m)}_{i,\overline{p}}$ for $i\in \Iw$ with $\tau i =i = \wb i$, are not monomials in the Chevalley generator $B_i$ and in addition depend on a parity $\ov{p} \in \Z_2$. The $\imath$divided powers $B^{(m)}_{i,\overline{p}}$ in the universal $\imath$quantum group $\tUi$ is formulated in \eqref{eq:iDPodd}--\eqref{eq:iDPev}, whose central reductions give us the version of $\imath$divided powers in $\Ui$ used in \cite{BW18a, BeW18, CLW18}.

 The $\imath$divided powers were essential in the formulation \cite{CLW18} of a distinguished $\imath$Serre relations in $\tUi$: for $i\neq j\in \Iw$ such that $\tau i = i =\wb i$,
\begin{align}
  \label{iSe}
\sum_{r+s =1-a_{ij}} (-1)^r  B_{i,\overline{p}}^{(r)}B_j B_{i,\overline{p} +\overline{a_{ij}}}^{(s)} &=0.
\end{align}
The $\imath$Serre relations in different forms for $\Ui$ associated with small values of Cartan integers $a_{ij}$ were known earlier \cite{Le02, Ko14, BK19}, but the expressions of $\imath$Serre relations in terms of monomials of Chevalley generators $B_i$ are getting quickly too cumbersome to be written down explicitly as $|a_{ij}|$ grows.
The $\imath$Serre relations \eqref{iSe} then led to a Serre presentation for quasi-split $\imath$quantum group $\Ui$ of arbitrary Kac-Moody type \cite{CLW18}.

\subsection{Goal}

The goal of this paper is to formulate and establish in full generality the Serre-Lusztig relations of $\imath$quantum groups associated to the $\imath$Serre relation \eqref{iSe}. We shall mainly work with $\tUi$ in this paper, and the Serre-Lusztig relations for $\Ui$ take the same form as for $\tUi$. The Serre-Lusztig relations allow us to formulate (partly conjectural) braid group symmetries for $\tUi$ and $\Ui$.

The main results in this paper provide another example to reinforce  a general expectation (first advocated in \cite{BW18a}) that most of the basic constructions for quantum groups admit (possibly highly nontrivial) natural generalizations in the setting of $\imath$quantum groups.

\subsection{Main results}
  \label{subsec:main}

We shall formulate and establish Serre-Lusztig relations for $\tUi$ in two stages by starting with those of minimal degrees.

For various structures of $\imath$quantum groups, it is conceptual and essential to work with $\imath$divided powers.
In an approach toward canonical basis arising from quantum symmetric pairs of Kac-Moody type \cite{BW18c}, 3 different types of $\imath$divided powers associated to $j \in \Iwhite$ are constructed, depending on
\[
\text{ (i) } \tau j = j =\wb j;
\quad
\text{ (ii) } \tau j \neq j;
\quad
\text{ (iii) } \tau j = j \neq \wb j.
\]
The $\imath$divided powers in cases (ii)-(iii) are denoted by $B_i^{(m)}$, for $m\ge 0$, and they do not depend on a parity as in Case (i) as described above. 

{\em By convention, we will use $B_{j, \ov{t}}^{(m)}$ to denote any of the above $\imath$divided powers in the settings where the conditions (i)-(iii) on $j$ are not specified; the index $\ov{t}$ is ignored in Cases (ii)-(iii).}

\begin{customthm}{{\bf A}}
[Serre-Lusztig relations  of minimal degree]
For any $i \neq j\in \Iw$ such that $\tau i = i =\wb i$, the following identities hold in $\tUi$:
\begin{align}
\sum_{r+s=1-na_{ij}} (-1)^r B^{(r)}_{i,\ov{p}} B_j^n B_{i,\ov{p}+\ov{na_{ij}}}^{(s)} &=0, \quad \text{for }n \ge 0,
  \label{eq:SerreBn12}
\\
\sum_{r+s=1-na_{ij}} (-1)^r B^{(r)}_{i,\ov{p}} B_{j, \ov{t}}^{(n)} B_{i,\ov{p}+\ov{na_{ij}}}^{(s)} &=0, \quad \text{for }n\ge 0.
  \label{eq:SerreBn10}
\end{align}
\end{customthm}
The identity \eqref{eq:SerreBn12} in Theorem~{\bf A} is a combination of Theorem~\ref{thm:minLS} (for $\tau j = j =\wb j$), Propositions~\ref{prop:jnotfixed} (for $\tau j \neq j$), and Proposition~\ref{prop:jfixed2} (for $\tau j = j \neq \wb j$). 
Amazingly, the relation \eqref{eq:SerreBn10} takes the same form as \eqref{eq:minSLQG} for the usual quantum groups. We view \eqref{eq:SerreBn10} to be more fundamental than \eqref{eq:SerreBn12} as it is valid at the level of  integral forms for (modified) $\imath$quantum groups, cf. \cite{BW18b}.

A version of Serre-Lusztig relations of minimal degree with supporting examples was proposed earlier by Baseilhac and Vu \cite{BaV14, BaV15} in the framework of tridiagonal pairs, for certain split $\imath$quantum groups (with $a_{ij}=-2, -1$, respectively); a more explicit form of these relations was conjectured in \cite{BaV14} (with $a_{ij}=-2$) and subsequently established in \cite{Ter18}; see Remark~\ref{rem:BVT}. The expressions of these relations in these works are in terms of monomials in Chevalley generators $B_i$ and look rather cumbersome. The Serre-Lusztig relations can be useful in the further study of $\imath$quantum groups with $q$ being a root of 1 (cf. \cite{BS21}), which according to  \cite{BaV14, BaV15} may play a central role in the identification of the symmetries of the Hamiltonian of the XXZ open spin chain.

The effort to understand and formulate the connections between \eqref{eq:SerreBn12} and \eqref{eq:SerreBn10} has led to the following theorem, which is an immediate consequence of Theorem~{\bf A} and Proposition~\ref{prop:SLinduct}.

\begin{customthm}{\bf B}
[Non-standard Serre-Lusztig relations for $\tUi$]
For any $i \neq j\in \Iw$ such that $\btau i=i=\wb i$ and $n,t \in \Z_{\ge 0}$, the following identities hold in $\tUi$:
\begin{align}
\sum_{r+s=1-na_{ij}+2t} (-1)^r B^{(r)}_{i,\ov{p}} B_j^n B_{i,\ov{p}+\ov{na_{ij}}}^{(s)} &=0,
  \label{eq:Serre2t}
\\
\sum_{r+s=1-na_{ij}+2t} (-1)^r B^{(r)}_{i,\ov{p}} B_j^{(n)} B_{i,\ov{p}+\ov{na_{ij}}}^{(s)} &=0.
  \label{eq:Serre1t}
\end{align}
(Note there is no $q$-powers involved in these identities in contrast to \eqref{eq:fij}--\eqref{eq:LS}.)
\end{customthm}



By going through numerous examples, we manage to guess explicit formulas for elements $\tf_{i,j; n,m,\ov{p},\ov{t},e}$ in $\tUi$, which is a proper $\imath$-analogue of $f_{i,j; n,m,e}^- \in \U^-$, the $F$-version of $f_{i,j; n,m,e}^+$; actually, $\tf_{i,j; n,m,\ov{p},\ov{t},e}$ has $f_{i,j; n,m,e}^-$ as its leading term. In addition, $\tf_{i,j;1,1-a_{ij},\ov{p},\ov{t},e}$ coincides with LHS of \eqref{iSe} and $\tf_{i,j;n,1-na_{ij},\ov{p},\ov{t},e}$ coincides with LHS of \eqref{eq:Serre2t}--\eqref{eq:Serre1t}.
It turns out the definition of $\tf_{i,j;n,m,\ov{p},\ov{t},e}$ depends on the parity of $m-na_{ij}$, cf. \eqref{eq:m-aodd}--\eqref{eq:m-aeven}.
We have the following recursion in $\tUi$ which formally looks like a mixture of the recursion formulas \eqref{eq:Lrec}--\eqref{eq:LrecF} in $\U$. 

\begin{customthm}{\bf C}   [Theorem~\ref{thm:recursion}]
For $i\neq j\in \Iw$ such that  $\btau i=i=\wb i$, $\ov{p}, \ov{t} \in\Z_2$, $n\ge 0$, and $e=\pm1$, we have
\begin{align}
&q_i^{-e(2m+na_{ij})}  B_i\tf_{i,j;n,m,\ov{p},\ov{t},e}-\tf_{i,j;n,m,\ov{p},\ov{t},e}B_i\\ \notag
&\quad = -[m+1]_{i} \tf_{i,j;n,m+1,\ov{p},\ov{t},e}
+[m+na_{ij}-1]_{i} q_i^{1-e(2m+na_{ij}-1)} \tk_i \tf_{i,j;n,m-1,\ov{p},\ov{t},e}.
\end{align}
\end{customthm}

The following generalizes Theorem~{\bf A}.

\begin{customthm}{\bf D}  [Serre-Lusztig relations for $\tUi$; see Theorem~\ref{thm:f=f'=0}]
Let $i\neq j\in \Iw$ such that  $\btau i=i=\wb i$, $\ov{p}\in\Z_2$, $n\ge 0$, and $e=\pm1$.
Then, for $m<0$ and $m>-na_{ij}$, we have
\begin{align}
\tf_{i,j;n,m,\ov{p},\ov{t},e}=0.
\end{align}
\end{customthm}

An anti-involution $\sigma_\imath$ for $\tUi$ constructed in \cite[Proposition 3.13]{BW18c} allows us to obtain an additional family of Serre-Lusztig relations for $\tUi$ involving new elements $\tf'_{i,j;n,m,\ov{p},\ov{t},e}$; for details see Theorems~\ref{thm:recursion}--\ref{thm:f=f'=0}.

Theorem~{\bf A} through Theorem~{\bf D} remain valid over $\Ui =\Ui_\bvs$, once we replace $\tk_i$ by the scalar $\vs_i$ in all relevant places and use the version of $\imath$divided powers in \eqref{eq:iDPoddUi}--\eqref{eq:iDPevUi}.

\subsection{Our approach}

The proof of Theorem~{\bf A} is much more challenging than its counterpart in the quantum group setting.

We show that the two identities \eqref{eq:SerreBn12} and \eqref{eq:SerreBn10} in Theorem~{\bf A} are equivalent. In case when $\tau j \neq j$, as the $\imath$divided powers of $B_j$ are standard, this equivalence is trivial.
However, in cases when $\tau j =j =\wb j$ and $\tau j =j \neq \wb j$, the $\imath$divided powers of $B_j$ have lower order terms. We show by using a key Proposition~\ref{prop:SLinduct} that the above identity \eqref{eq:SerreBn12} implies the following non-standard Serre-Lusztig relations (equivalent to Theorem {\bf B}): for $n, t \in \Z_{\ge 0}$,
\begin{align}
    \label{eq:nonstd}
\sum_{r+s=1-na_{ij}} (-1)^r B^{(r)}_{i,\ov{p}} B_j^{n-2t} B_{i,\ov{p}+\ov{na_{ij}}}^{(s)}  =0.
\end{align}
Now the identity \eqref{eq:SerreBn10} in Theorem~{\bf A} follows from \eqref{eq:SerreBn12} and \eqref{eq:nonstd}.

It remains to prove \eqref{eq:SerreBn12}. The proof is long and computational, and it follows a similar strategy in \cite{CLW18}  used in the proof of $\imath$Serre relation \eqref{iSe} (which is a special case of \eqref{eq:SerreBn12} at $n=1$). That is,
we use the expansion formulas of $\imath$divided powers into PBW basis of quantum $\mathfrak{sl}_2$ from \cite{BeW18} to reduce the proof to certain $q$-binomial identities; the Serre-Lusztig relations from quantum groups will be used as well.  Recall \cite{CLW18} we reduce the proof of the $\imath$Serre relation \eqref{iSe} to a $q$-identity with 3 auxiliary variables by such PBW expansion; instead of proving this $q$-identity directly (which we didn't know how), we deduce it from more general identities involving a function $G$ in 6 auxiliary variables (which admits simpler recursions).  Almost miraculously, the $q$-identity arising from our current reduction from \eqref{eq:SerreBn12}, which is much more involved than \cite{CLW18}, also follows from the same collection of identities involving $G$.

For 3 types of $\imath$divided powers for $B_j$ (i.e., (i)-(iii) in \S\ref{subsec:main}), the details of the proofs of \eqref{eq:SerreBn12} are largely the same with some differences. We give the complete details in case (i), and explain the differences in cases (ii)-(iii).

The main difficulty of Theorem~{\bf C} lies in its precise formulation (including guessing the formulas for $\tf_{i,j;n, m,\ov{p},\ov{t},e}$); its proof requires only routine though lengthy computations.

Observe that the base case $\tf_{i,j;n, 1-na_{ij},\ov{p},\ov{t},e}=0$ in Theorem~{\bf D} is exactly the Serre-Lusztig relation of minimal degree in Theorem~{\bf A}. Theorem~{\bf D} in general now follows readily from the recursion formulas in Theorem~{\bf C}.

\subsection{Applications} 

Keeping in mind Lusztig's formulas for braid group symmetries on $\U$ and connections to Serre-Lusztig relations \cite[Part VI]{Lu93}, the Serre-Lusztig relations for $\tUi$ and $\Ui$ suggest natural formulas involving $\tf_{i,j;n,-na_{ij},\ov{p},\ov{t},e}$ and $\tf'_{i,j;n,-na_{ij},\ov{p},\ov{t},e}$ in \eqref{eq:yy}--\eqref{eq:yy2} for braid group symmetries $\T_{i,e}'$ and $\T_{i,e}''$ on $\tUi$ and $\Ui$; see Conjecture~\ref{conj:braid}. For earlier works on braid group actions (associated to the underlying restricted root system) on $\Ui$ and $\tUi$, see \cite{KP11, LW19b, D19} for finite type and  see \cite{BaK20} for $q$-Onsager algebra (i.e., $\Ui$ of split affine type $A_1$).

Even for finite type, no braid group action on $\Ui$-modules is available, in contrast to the quantum group setting  \cite[Chapter 5]{Lu93}. This makes it difficult to verify directly the conjecture that $\T_{i,e}'$ and $\T_{i,e}''$ are algebra automorphisms of $\tUi$.
In a subsequent work, we shall develop further the $\imath$Hall algebra approach (cf. \cite{LW19b}) to establish Conjecture~\ref{conj:braid} for {\em quasi-split} $\imath$quantum groups $\tUi$ and $\Ui$ of Kac-Moody type.

\subsection{Organization}

The paper is organized as follows.
In the preliminary Section~\ref{sec:QSP}, we set up notations for Drinfeld doubles, $\imath$quantum groups, and $\imath$divided powers. In Section~\ref{sec:induction}, we formulate a key induction procedure, which will be used repeatedly.

In Section~\ref{sec:minimalSL} (and respectively, Section~\ref{sec:minimalSL2}), we establish the Serre-Lusztig relation \eqref{eq:SerreBn12} in Theorem~{\bf A} for $j\in \I$ with $\tau j= j =\wb j$ (and respectively, $\tau j \neq j$ or  $\tau j =j=\wb j$); parts of the proof are postponed to Appendix~\ref{App:A}. We then complete the proofs of Theorem~{\bf A} and Theorem~{\bf B}.

In Section~\ref{sec:SerreL1}, we prove the general Serre-Lusztig relations (Theorem~{\bf D}) by first establishing the recursive formulas (Theorem~{\bf C}).
In Appendix~\ref{App:B}, we prove some combinatorial $q$-binomial identities used in the proof of the general Serre-Lusztig relations in \S\ref{subsec:even}--\ref{subsec:odd}.

\subsection*{Acknowledgments.}
XC is partially supported by the National Natural Science Foundation of China grant No. 11601441 and the Fundamental Research Funds for the Central Universities grant No. 2682020ZT100. WW is partially supported by NSF grant DMS-1702254 and DMS-2001351. We thank University of Virginia and Institute of Mathematics, Academia Sinica (Taipei), for hospitality and excellent working environment.  We thank the anonymous referees for helpful comments and suggestions.

\section{Quantum symmetric pairs and $\imath${}quantum groups}
  \label{sec:QSP}

In this section, we recall the definitions of $\imath${}quantum groups and quantum symmetric pairs (QSP). We introduce universal $\imath${}quantum groups as subalgebras of Drinfeld doubles and $\imath$divided powers for $\tUi$. Then we review some $q$-binomial identities from \cite{CLW18}.

\subsection{Quantum groups and Drinfeld doubles}


Given a Cartan datum  $(\I,\cdot)$, we have a \emph{root datum} of type $(\I,\cdot)$ \cite[1.1.1, 2.2.1]{Lu93}, which consists of
\begin{itemize}
\item[(a)] two finitely generated free abelian groups $Y,X$ and a perfect bilinear pairing $\langle\cdot,\cdot\rangle:Y\times X\rightarrow\Z$;
\item[(b)] an embedding $\I\subset X$ ($i\mapsto \alpha_i$) and an embedding $\I\subset Y$ ($i\mapsto h_i$) such that $\langle h_i,\alpha_j\rangle =2\frac{i\cdot j}{i\cdot i}$ for all $i,j\in \I$.
\end{itemize}
We assume that the root datum defined above is \emph{$X$-regular} and \emph{$Y$-regular}, that is, $\{\alpha_i\mid i\in \I\}$ is linearly independent in $X$ and $\{h_i\mid i\in \I\}$ is linearly independent in $Y$. We further assume $Y$ is of the form in this paper
\begin{align} \label{eq:YY}
Y =(\oplus_{i\in \I} \Z h_i) \oplus Y',
\quad \text{ for a free abelian group } Y'.
\end{align}
For example, the $Y$ arising from a minimal realization of a generalized Cartan matrix is of this form.

%
%
Let $q$ be an indeterminate, and denote
\begin{align}
\epsilon_i :=\frac{i\cdot i}{2},
\quad
q_i:=q^{\epsilon_i}, \qquad \forall i\in \I.
\end{align}
The matrix
\[
C= (a_{ij})_{i,j\in \I} =(\langle h_i, \alpha_j\rangle)_{i,j\in \I}
\]
is a \emph{generalized Cartan matrix}.
The matrix $DC$ is symmetric, where the symmetrizer $D :=\diag(\epsilon_i\mid  i\in \I)$ is a diagonal matrix.
For $n,m\in \Z$ with $m\ge 0$, we denote the $q$-integers and $q$-binomial coefficients as
\begin{align*}
\begin{split}
[n] =\frac{q^n-q^{-n}}{q-q^{-1}},
\quad
[m]! =\prod_{i=1}^m [i],
\quad
\qbinom{n}{d}  &=
\begin{cases}
\frac{[n][n-1]\ldots [n-d+1]}{[d]!}, & \text{ if }d \ge 0,
\\
0, & \text{ if }d<0.
\end{cases}
\end{split}
\end{align*}
We denote by $[n]_{q_i}$ and $\qbinom{n}{d}_{q_i}$, or simply $[n]_{i}$ and $\qbinom{n}{d}_{i}$,  the variants of $[n]$ and $\qbinom{n}{d}$ with $q$ replaced by $q_i$. For  any $i\neq j\in \I$, define the following polynomial in two (noncommutative) variables
\begin{align*}
S_{ij}(x,y)=\sum_{r=0}^{1-a_{ij}} (-1)^r \qbinom{1-a_{ij}}{r}_{i}  x^{r}yx^{1-a_{ij}-r}.
\end{align*}

Let $\K$ be a field of characteristic $0$.
Associated to a root datum $(Y,X,\langle\cdot,\cdot\rangle,\dots)$ of type $(\I,\cdot)$,
the {\em Drinfeld double} $\tU$ is the associative $\K(q)$-algebra with generators $E_i,F_i,K_h,K_h'$ for all $i\in \I$,  
subject to the following relations: for $h, h' \in Y, i, j \in \I$,
\begin{align}
K_hK_{h'} &=K_{h'}K_{h},\; K_hK'_{h'}=K'_{h'}K_h,\; K'_hK'_{h'}=K'_{h'}K'_{h},
\label{Q1}
\\
K_hE_i &=q^{\langle h,\alpha_i\rangle} E_i K_h, \qquad
 K_hF_i=q^{-\langle h,\alpha_i\rangle}F_i K_h,
\label{eq:EK}
\\
K'_hE_i &=q^{-\langle h,\alpha_i\rangle} E_i K'_h, \qquad
 K'_hF_i=q^{\langle h,\alpha_i\rangle}F_i K'_h,
 \label{eq:K2}
 \\
 [E_i,F_j] &=\delta_{ij}\frac{\tK_{i}-\tK_{i}'}{q_i-q_i^{-1}},
\quad \text{ where }\widetilde{K}_i :=K_{h_i}^{\epsilon_i},
\label{Q4}\\
S_{ij}(E_i,E_j) &=0=S_{ij}(F_i,F_j), \quad \forall i \neq j\in \I.\label{q serre}
\end{align}

Let
\[
F_i^{(n)} =F_i^n/[n]_{i} ^!, \quad E_i^{(n)} =E_i^n/[n]_{i} ^!, \quad \text{ for } n\ge 1 \text{ and  } i\in \I.
\]
Then the $q$-Serre relations \eqref{q serre} above
can be rewritten  as follows: for $i\neq j \in \I$,
\begin{align*}
\sum_{r=0}^{1-a_{ij}} (-1)^r  E_i^{(r)} E_j E_i^{({1-a_{ij}-r})}=0,
\\
\sum_{r=0}^{1-a_{ij}} (-1)^r  F_i^{(r)}F_j F_i^{(1-a_{ij}-r)}=0 .
\end{align*}

Note that $\tK_i\tK'_i$ are central in $\tU$ for any $i\in \I$.  The comultiplication $\Delta: \widetilde{\U} \rightarrow \widetilde{\U} \otimes \widetilde{\U}$ is defined as follows:
\begin{align}  \label{eq:Delta}
\begin{split}
\Delta(E_i)  = E_i \otimes 1 + \tK_i \otimes E_i, & \quad \Delta(F_i) = 1 \otimes F_i + F_i \otimes \tK_{i}', \\
 \Delta(\tK_h) = \tK_h \otimes \tK_h, & \quad \Delta(\tK_h') = \tK_h' \otimes \tK_h'.
 \end{split}
\end{align}

Analogously as for $\tU$, the quantum group $\bU$ is defined to be the $\K(q)$-algebra generated by $E_i,F_i, K_h^{\pm 1}$, for all $i\in \I, h\in Y$, subject to the relations \eqref{q serre}, \eqref{Q1}--\eqref{eq:EK}  (with the relations involving $K'_{h'}$ ignored),
and with \eqref{Q4} replaced by $[E_i,F_j]=\delta_{ij}\frac{K_{i}-K_{i}^{-1}}{q_i-q_i^{-1}}$.
The comultiplication $\Delta$ for $\U$ is modified from \eqref{eq:Delta} 
 with $\tK_i$ and $\tK_i'$ replaced by $K_{i}$ and $K_{i}^{-1}$, respectively. (Beware that our $K_i$ has a different meaning from $K_i \in \U$ in \cite{Lu93}.)


Let $\widetilde{\bU}^+$ be the subalgebra of $\widetilde{\bU}$ generated by $E_i$ $(i\in \I)$, $\widetilde{\bU}^0$ be the subalgebra of $\widetilde{\bU}$ generated by $\tK_i, \tK_i'$ $(i\in \I)$, and $\widetilde{\bU}^-$ be the subalgebra of $\widetilde{\bU}$ generated by $F_i$ $(i\in \I)$, respectively.
The subalgebras $\bU^+$, $\bU^0$ and $\bU^-$ of $\bU$ are defined similarly. Then both $\widetilde{\bU}$ and $\bU$ have triangular decompositions:
\begin{align}
\widetilde{\bU} =\widetilde{\bU}^+\otimes \widetilde{\bU}^0\otimes\widetilde{\bU}^-,
\qquad
\bU &=\bU^+\otimes \bU^0\otimes\bU^-.
\end{align}
Clearly, ${\bU}^+\cong\widetilde{\bU}^+$, ${\bU}^-\cong \widetilde{\bU}^-$, and ${\bU}^0 \cong \widetilde{\bU}^0/(\tK_h \tK_h' -1 \mid   h\in Y)$. Note that $\U^+ =\cup_\mu \U^+_\mu$ is $\Z_{\ge 0} \I$-graded by setting $\deg E_i=\alpha_i$ for any $i\in\I$, where $\U^+_\mu:=\{x\in\U^+\mid \deg x=\mu\}$.

Denote by $r_{i}: \U^+ \rightarrow \U^+$ the unique $\K(q)$-linear maps \cite{Lu93}  such that
\begin{align}  \label{eq:rr}
r_{i}(1) = 0, \quad r_{i}(E_{j}) = \delta_{ij},
\quad r_{i}(xx') = xr_{i}(x') + q^{i \cdot \mu'}r_{i}(x)x',
\end{align}
for all $x \in \U^+_{\mu}$ and $x' \in \U^+_{\mu'}$. 


\subsection{The algebra $\Udot$}
   \label{subsection Udot}

Recall \cite[23.1]{Lu93} that  the modified form of $\bU$, denoted by $\Udot$, is a $\K(q)$-algebra (without 1)  generated by $\oldone_\la, E_i \oldone_\la$, $F_i \oldone_\la$, for $i\in \I, \la \in X$, where $\oldone_\la$ are orthogonal idempotents.
Note that $\Udot$ is naturally a $\bU$-bimodule \cite[23.1.3]{Lu93}, and in particular we have
\begin{align*}
K_h\oldone_\la=\oldone_\la K_h &=q^{\langle h,\la\rangle}\oldone_\la, \; \forall h\in Y.
\end{align*}

We have the mod $2$ homomorphism $\Z \rightarrow \Z_2, k \mapsto \ov k$, where $\Z_2=\{\bar{0},\bar{1}\}$. Let us fix an $i\in \I$.  Define
\begin{equation}
\Udot_{i,\ev}:=\bigoplus_{\la:\, \langle h_i,\la\rangle \in 2\Z}\,\Udot \oldone_{\la},
\qquad
\Udot_{i,\odd}:=\bigoplus_{\la:\, \langle h_i,\la\rangle \in1+2\Z}\, \Udot\oldone_{\la}.
\end{equation}
Then $\Udot= \,\Udot_{i,\ev}\oplus{\Udot_{i,\odd}}$.

For our later use, with $i\in \I$ fixed once for all, we need to keep track of the precise value $\langle h_i,\la\rangle$ in an idempotent $\oldone_\la$ but do not need to  know which specific weights $\la$ are used. Thus it is convenient to introduce the following notation
\begin{align}
 \label{eq:1star}
\onestar_m =\onestar_{i,m}, \qquad \text{ for }m\in \Z,
\end{align}
to denote an idempotent $\oldone_\la$ for some $\la\in X$ such that $\m=\langle h_i,\la\rangle$.
In this notation, the identities in \cite[23.1.3]{Lu93} can be written as follows: for any $\m\in\Z$, $a,b\in\Z_{\geq0}$, and $i\neq j\in \I$,
\begin{align}
E_i^{(a)}\onestar_{i,\m} &=\onestar_{i,\m+2a} E_i^{(a)},\quad F_i^{(a)}\onestar_{i,\m}=\onestar_{i,\m-2a} F_i^{(a)};
\label{eqn: idempotent Ei Fi}\\
E_j\onestar_{i,\m} &=\onestar_{i,\m+a_{ij}}E_j,  \qquad F_j\onestar_{i,\m}=\onestar_{i,\m-a_{ij}} F_j;
  \label{eqn: idempotent Ej Fj}\\
F_{i}^{(a)} E_i^{(b)}\onestar_{i,\m} &= \sum_{d=0}^{\min\{a,b\}} \qbinom{a-b-\m}{d}_{i}  E_i^{(b-d)} F_i^{(a-d)}\onestar_{i,\m};
  \label{eqn:commutate-idempotent3}\\
E_i^{(a)} F_i^{(b)}\onestar_{i,\m} &=\sum_{d=0}^{\min\{a,b\}} \qbinom{a-b+\m}{d}_{i}  F_i^{(b-d)} E_i^{(a-d)}\onestar_{i,\m}
  \label{eqn:commutate-idempotent4}.
\end{align}
From now on, we shall always drop the index $i$ to write the idempotents as $\onestar_m$.

\begin{rem}
  \label{rem:u=0}
If $u\in\U$ satisfies $u\onestar_{2k-1}=0$ for all possible idempotents $\onestar_{2k-1}$ with $k\in\Z$ (or respectively,  $u\onestar_{2k}=0$ for all possible $\onestar_{2k-1}$ with $k\in\Z$), then $u=0$.
\end{rem}
\subsection{The $\imath${}quantum groups $\tUi$ and $\Ui$}
  \label{subsec:iQG}

Let $\tau$ be an involution of the Cartan datum $(\I, \cdot)$; we allow $\tau ={\rm id}$.
Let $\I_{\bullet} \subset \I$ be a Cartan subdatum of {\em finite type}. Let $W_{\I_\bullet}$ be the Weyl subgroup for $(\I_{\bullet}, \cdot)$ with $w_{\bullet}$ as  its longest element. Let $\rho^\vee_{\bullet}$ be half the sum of all positive coroots associated to $(\I_{\bullet}, \cdot)$. We shall denote
\begin{equation}
\label{eq:white}
 \I_{\circ} = \I \backslash \I_{\bullet}.
\end{equation}

A pair $(\I_{\bullet}, \tau)$ is called {\em admissible} (cf. \cite[Definition~2.3]{Ko14}) if 
\begin{itemize}
	\item	[(1)]	$\tau (\I_{\bullet}) = \I_{\bullet}$;
	\item	[(2)]	The action of $\tau$ on $\I_{\bullet}$ coincides with the action of $-w_{\bullet}$;
	\item	[(3)]	If $j \in \I_{\circ}$ and $\tau(j) = j$, then $\langle \rho^\vee_{\bullet}, j' \rangle \in \Z$.
\end{itemize}
All pairs $(\I_{\bullet}, \tau)$ considered in this paper are admissible.

Following and slightly generalizing \cite{LW19a}, we define a {\em universal $\imath$quantum group} $\widetilde{\bU}^\imath$ to be the $\K(q)$-subalgebra of the Drinfeld double $\tU$ generated by $E_\ell, F_\ell, \tK_\ell, \tK'_\ell$, for $\ell \in \Iblack$, and
\[
B_i= F_i + \T_{\wb} (E_{\btau i}) \tK_i',
\qquad \tk_i= \tK_i \tK_{\btau i}', \quad \forall i \in \Iwhite.
\]
Here $\T_w$, for $w\in W_{\Iblack}$, corresponds to $\T''_{w,+1}$ in \cite[Ch. 37]{Lu93}; see \cite[Proposition~ 2.1]{BW18c}.
Then $\tUi$ is a coideal subalgebra of $\tU$ in the sense that $\Delta: \tUi \rightarrow \tUi\otimes \tU$. 

\begin{lem}
The elements $\tK_\ell \tK'_\ell$ ($\ell \in \Iblack$),
$\tk_i$ (for $\btau i=i \in \Iwhite$) and $\tk_i \tk_{\btau i}$  (for $\btau i\neq i \in \Iwhite$) are central in $\tUi$.
\end{lem}

Let $\bvs=(\vs_i)\in  (\K(q)^\times)^{\Iwhite}$ be such that $\vs_i=\vs_{\btau i}$ whenever $a_{i, \btau i}=0$.
Let $\Ui:=\Ui_{\bvs}$ be the $\K(q)$-subalgebra of $\bU$ generated by $E_\ell, F_\ell, \tK_\ell^{\pm 1}$, for $\ell \in \Iblack$, and
\[
B_i= F_i+\vs_i \T_{\wb} (E_{\btau i})K_i^{-1}\; (\forall i \in \Iwhite),
\quad
k_j= K_jK_{\btau j}^{-1} \; (\forall \btau j\neq j \in \Iwhite).
\]
It is known \cite{Le99, Ko14} that $\bU^\imath$ is a right coideal subalgebra of $\bU$, and $(\bU,\Ui)$ is called a \emph{quantum symmetric pair} ({\em QSP} for short), as they specialize at $q=1$ to a symmetric pair.

The following is a $\tUi$-variant of an anti-involution $\sigma_\imath$ on $\Ui$ in \cite[Proposition 3.13]{BW18c}. 

 \begin{lem}
   \label{lem:isigma}
There exists a $\K(q)$-linear anti-involution $\sigma_\imath$ of the algebra $\tUi$ such that
\begin{align*}
F_\ell \mapsto F_{\tau \ell},
\;\;
E_\ell \mapsto E_{\tau \ell},
\;\;
K_\ell \mapsto K'_{\tau \ell} \;\; (\forall \ell \in \Iblack);
\quad
B_j \mapsto B_{\tau j},
\;\;
\tk_j \mapsto \tk_j \;\; (\forall j \in \Iwhite).
\end{align*}
 \end{lem}

\subsection{The $\imath$divided powers}

For  $i\in \I$ with $\btau i\neq i$,  following \cite{BW18b},  we define the {\em $\imath${}divided powers} of $B_i$ to be
\begin{align*}
  B_i^{(m)}:=B_i^{m}/[m]_{i}^!, \quad \forall m\ge 0, \qquad (\text{if } i \neq \tau i).
\end{align*}

For $i\in \Iw$ with $\btau i= i =\wb i$, we defined the {\em $\imath${}divided powers} of $B_i$ in $\Ui =\Ui_{\bvs}$ \cite{CLW18} (which generalizes \cite{BW18b, BeW18} where $\vs_i=q_i^{-1}$) to be
\begin{align}
\ff_{i,\odd}^{(m)} &=\frac{1}{[m]_{i}^!}\left\{ \begin{array}{ccccc} B_i\prod_{j=1}^k (B_i^2-q_i\vs_i[2j-1]_{i}^2 ) & \text{if }m=2k+1,\\
\prod_{j=1}^k (B_i^2-q_i\vs_i[2j-1]_{i}^2) &\text{if }m=2k; \end{array}\right.
  \label{eq:iDPoddUi} \\
  \notag\\
\ff_{i,\ev}^{(m)} &= \frac{1}{[m]_{i}^!}\left\{ \begin{array}{ccccc} B_i\prod_{j=1}^k (B_i^2-q_i\vs_i[2j]_{i}^2 ) & \text{if }m=2k+1,\\
\prod_{j=1}^{k} (B_i^2-q_i\vs_i[2j-2]_{i}^2) &\text{if }m=2k. \end{array}\right.
 \label{eq:iDPevUi}
\end{align}
Replacing $\vs_i$ by $\tk_i$ and abusing notations, we define the {\em $\imath${}divided powers} of $B_i$ in $\tUi$ to be
\begin{align}
\ff_{i,\odd}^{(m)} &=\frac{1}{[m]_{i}^!}\left\{ \begin{array}{ccccc} B_i\prod_{j=1}^k (B_i^2-q_i\tk_i[2j-1]_{i}^2 ) & \text{if }m=2k+1,\\
\prod_{j=1}^k (B_i^2-q_i\tk_i[2j-1]_{i}^2) &\text{if }m=2k; \end{array}\right.
  \label{eq:iDPodd} \\
  \notag\\
\ff_{i,\ev}^{(m)} &= \frac{1}{[m]_{i}^!}\left\{ \begin{array}{ccccc} B_i\prod_{j=1}^k (B_i^2-q_i\tk_i[2j]_{i}^2 ) & \text{if }m=2k+1,\\
\prod_{j=1}^{k} (B_i^2-q_i\tk_i[2j-2]_{i}^2) &\text{if }m=2k. \end{array}\right.
 \label{eq:iDPev}
\end{align}
We set $B_{i,\ov{p}}^{(m)}=0$ for any $m<0.$

These $\imath$divided powers in $\tUi$ satisfy the following recursive relations:
\begin{align}
\label{lem:dividied power}
B_{i}B_{i,\ov{p}}^{(r)}=\left\{\begin{array}{llll}
&[r+1]_{i}B_{i,\ov{p}}^{(r+1)}                                 & \text{if}\ \ov{p}\neq \ov{r};\\
&[r+1]_{i}B_{i,\ov{p}}^{(r+1)}+q_i\tk_i[r]_{i}B_{i,\ov{p}}^{(r-1)} & \text{if}\ \ov{p}= \ov{r}.
\end{array}\right.
\end{align}
(The recursive relations for $\imath$divided powers in $\Ui$ are obtained by substituting $\tk_i$ with $\vs_i$.)

\begin{rem}
The results in this paper are formulated for $\tUi$, and their counterparts for $\Ui$ can be obtained (with the same proofs) by the simple substitution of $\tk_i$ with $\vs_i$.
\end{rem}

The $\imath$Serre relation in $\tUi$ below formally takes the same form as for $\Ui$ \cite[(3.9)]{CLW18}. No additional condition on $j\in \Iw$ is imposed, thanks to \cite[Remark~ 3.5]{CLW18}.

\begin{prop}
   \label{prop:iSerre}
The following $\imath$Serre relations hold in $\tUi$, for $j\neq i\in \Iw$ with $\wb i = \tau i =i$:
\begin{align*}
\sum_{r+s =1-a_{ij}} (-1)^r  B_{i,\overline{p}}^{(r)}B_j B_{i,\overline{p} +\overline{a_{ij}}}^{(s)} &=0.
\end{align*}
\end{prop}

\begin{proof}

Denote the LHS of the identity in $\tUi$ in the proposition by $L$.
Let $\bvs=(\vs_\ell)\in  (\K(q)^\times)^{\Iwhite}$ be such that $\vs_\ell=\vs_{\btau \ell}$ whenever $a_{\ell, \btau \ell}=0$. By a base change, all the algebras in this proof will be assumed to be over an extension field of $\K(q)$ which includes $\sqrt[\leftroot{-2}\uproot{2}2\epsilon_i]{\vs_\ell}$, for $\ell \in \I$.
By \eqref{eq:YY} there is a quotient morphism $\pi: \tU \rightarrow \U$, which sends  $F_\ell \mapsto F_\ell, E_\ell \mapsto \sqrt{\vs_\ell} E_\ell, K_{h_\ell} \mapsto   \sqrt[\leftroot{-2}\uproot{2}2\epsilon_i]{\vs_\ell} K_{h_\ell}, K'_{h_\ell} \mapsto \sqrt[\leftroot{-2}\uproot{2}2\epsilon_i]{\vs_\ell}  K_{h_\ell}^{-1}, K_{h'} \mapsto  K_{h'}$, for $\ell \in \I, h' \in Y'$;
note in particular $\pi$ sends $\tK_\ell \mapsto \sqrt{\vs_\ell} K_\ell, \tK'_\ell \mapsto \sqrt{\vs_\ell} K_\ell^{-1}$, for $\ell \in \I$.
By restriction we obtain a morphism $\pi: \tUi \rightarrow \Ui_{\bvs}$ which sends $B_i= F_i + \T_{\wb} (E_{i}) \tK_i'$ to $B_i = F_i+\vs_i \T_{\wb} (E_{i})K_i^{-1}$. Then $\pi$ matches the corresponding $\imath$divided powers of $B_i$ and thus maps $L$ to 0 by the $\imath$Serre relation in $\Ui$ \cite[(3.9), Remark~ 3.5]{CLW18}.
As the scalar $\vs_i$ varies, we conclude that $L=0$.
(The argument above actually shows that the $\imath$Serre relations for $\Ui$ and $\tUi$ imply each other.)
\end{proof}

\begin{rem}
The $\imath$Serre relations for $\Ui$ in case $|a_{ij}| \le 3$  were known earlier in different forms \cite{Le02, BK19}.
\end{rem}

\subsection{$q$-binomial identities from \cite{CLW18}}

For $w, p_0, p_1, p_2\in \bbZ$ and $u,\ell  \in \bbZ_{\ge 0}$, we define (cf. \cite[(5.1)]{CLW18})
\begin{align}
\label{eq:Twxyu}
G(w,u, & \ell;p_0,p_1,p_2):=(-1)^{w}q^{u^2-wu+\ell u}\\
&\cdot \left\{ \sum_{\substack{b,c,e\geq0 \\ b+c+e=u}}
\sum^{\ell }_{\substack{t=0 \\ 2\mid(t+w-b) }}
q^{-t(\ell +u-1)-u(c+e)+2c+rp_0+2cp_1+2ep_2} \right.  \notag\\ \notag
&\qquad \cdot \qbinom{\ell }{t}_q\qbinom{w+t+p_0}{b}_q \qbinom{\frac{w+t-b}{2}+p_1}{c}_{q^2}\qbinom{\frac{w+t-b}{2}+p_2}{e}_{q^2}\\  \notag
&\quad -\sum_{\substack{b,c,e\geq0 \\ b+c+e=u}}
\sum^{\ell }_{\substack{t=0 \\ 2\nmid(t+w-b) }}
q^{-t(\ell +u-1)-(u-1)(c+e)+rp_0+2cp_1+2ep_2}\\ \notag
& \left. \qquad \cdot \qbinom{\ell }{t}_q\qbinom{w+t+p_0}{b}_q \qbinom{1+\frac{w+t-b-1}{2}+p_1}{c}_{q^2}\qbinom{\frac{w+t-b-1}{2}+p_2}{e}_{q^2} \right\}.  \notag
\end{align}

\begin{lem} [\text{\cite[Lemma 5.1, Theorem 5.6]{CLW18}}]
\label{lem:Gwuxp0p1p2}
For any $w,p_0,p_1,p_2,k\in \bbZ$ and $u,\ell  \in \bbZ_{\ge 0}$,  the following identities hold:
\begin{align}
G(w,u,\ell +1; p_0,p_1,p_2) &=q^{u}G(w,u,\ell ;p_0,p_1,p_2)-q^{u-2\ell }G(w+1,u,\ell ;p_0,p_1,p_2); \label{eq:Gx+1w}\\
G(w,u,\ell ;p_0,p_1,p_2) &=q^{4ku}G(w+2k,u,\ell ;p_0-2k,p_1-k,p_2-k); \label{eq:Gk}\\
G(w+1,u,\ell ;p_0,p_1,p_2) &=q^{-2u}G(w,u,\ell ;p_0+1,p_2,p_1+1); \label{eq:Godd}\\
G(w,u,\ell ; p_0,p_1,p_2) &=0, \quad \text{ if } \ell >0. \label{eq:G=0}
\end{align}
\end{lem}

For $p_1,p_2\in \bbZ$ and $u \in \bbZ_{\ge 0}$, we define (cf. \cite[(5.11)]{CLW18})
\begin{align*}
H(u; p_1,p_2):=\sum_{\stackrel{c,e\geq0}{c+e=u}}q^{2c+2cp_1+2ep_2}\qbinom{p_1}{c}_{q^2}\qbinom{p_2}{e}_{q^2}.
\end{align*}

\begin{lem}
[\text{\cite[Proposition 5.3]{CLW18}}]
  \label{lem:GH00}
For any $w,p_1,p_2\in \bbZ$ and $u \in \bbZ_{\ge 0}$, we have
$G(w,u,0;0,p_1,p_2) =H(u;p_1,p_2)$. In particular, $G(w,u,0;0,p_1,p_2)$ is independent of $w$.
\end{lem}

\section{A key induction}
   \label{sec:induction}

In this section, we establish an inductive formula on Serre-Lusztig relations, which will be used for several times in later sections.

\subsection{Recursions for $\imath$divided powers}

\begin{lem}
  \label{lem:B2r}
For $r \ge 0$, we have
{\small\begin{align*}
B^{(2)}_{i,\ov{0}} B^{(r)}_{i,\ov{0}} =
\begin{cases}
\qbinom{r+2}{2}_{i} B^{(r+2)}_{i,\ov{0}} +  \frac{[r]_{i}^2}{[2]_{i}} (q_i \tk_i) B^{(r)}_{i,\ov{0}}, & \text{ for $r$ even},
\\
\\
\qbinom{r+2}{2}_{i} B^{(r+2)}_{i,\ov{0}} + \frac{[r+1]_{i}^2}{[2]_{i}} (q_i \tk_i) B^{(r)}_{i,\ov{0}}, & \text{ for $r$ odd};
\end{cases}
\end{align*}
\begin{align*}
B^{(2)}_{i,\ov{1}} B^{(r)}_{i,\ov{1}} =
\begin{cases}
\qbinom{r+2}{2}_{i} B^{(r+2)}_{i,\ov{1}} + \frac{[r+1]_{i}^2-1}{[2]_{i}} (q_i \tk_i) B^{(r)}_{i,\ov{1}}, & \text{ for $r$ even},
\\
\\
\qbinom{r+2}{2}_{i} B^{(r+2)}_{i,\ov{1}} + \frac{[r]_{i}^2-1}{[2]_{i}} (q_i \tk_i) B^{(r)}_{i,\ov{1}}, & \text{ for $r$ odd}.
\end{cases}
\end{align*}
}
\end{lem}

\begin{proof}
Follows by a direct computation from the recursive definition of $\imath$divided powers \eqref{lem:dividied power}.
\end{proof}

\subsection{An induction step}

The following proposition will serve as a key induction step repeatedly in this paper.

\begin{prop}
  \label{prop:SLinduct}
Let $m\in \Z_{\ge 1}$ such that $m \not \equiv na_{ij} \pmod 2$ and let $i \in \Iw$. Suppose the following identity holds in $\tU$ for some $X\in \tU$:
\begin{align*}
\Xi : =\sum_{r+s=m} (-1)^r B^{(r)}_{i,\ov{p}}  X B_{i,\ov{p}+\ov{na_{ij}}}^{(s)} =0.
\end{align*}
Then the following identity holds in $\tU$:
\begin{align*}
\sum_{r+s=m+2} (-1)^r B^{(r)}_{i,\ov{p}}  X B_{i,\ov{p}+\ov{na_{ij}}}^{(s)} =0.
\end{align*}
\end{prop}
(The assumption here that $m \not \equiv na_{ij} \pmod 2$ is reasonable, as in application below we have  $m=1- na_{ij}$ as the starting point.)

\begin{proof}
Let us first outline the idea of the proof.
Using Lemma~\ref{lem:B2r} and the recursive definition of $\imath$divided powers we shall compute
$B^{(2)}_{i,\ov{p}} \Xi$, $\Xi B^{(2)}_{i,\ov{p}+\ov{na_{ij}}}$, and $B_{i} \Xi B_{i}$, and then a sum
\begin{equation}
  \label{eq:SZ}
S:= B^{(2)}_{i,\ov{p}} \Xi + \Xi B^{(2)}_{i,\ov{p}+\ov{na_{ij}}} - \frac{q_i^{m+1} +q_i^{-m-1}}{[2]_{i}} B_{i} \Xi B_{i}.
\end{equation}
We make the following claim.

{\bf Claim.} The following identity holds for some suitable scalars $\xi$:
\begin{align}
 \label{eq:SZZ}
S & =  \qbinom{m+2}{2}_{i} \sum_{r+s=m+2} (-1)^r B^{(r)}_{i,\ov{p}}  X B_{i,\ov{p}+\ov{na_{ij}}}^{(s)}
+ \xi \Xi.
\end{align}
Assuming the Claim, we conclude that $\sum_{r+s=m+2} (-1)^r B^{(r)}_{i,\ov{p}}  X B_{i,\ov{p}+\ov{na_{ij}}}^{(s)} =0$ from this identity and the assumption that $\Xi=0$,  establishing the proposition.

It remains to prove the identity \eqref{eq:SZZ} in the Claim. The proof is divided into Cases (1)-(4) below according to the parities $\ov{p}$ and $\ov{p}+\ov{na_{ij}}$;  the scalars $\xi$ depend on these parities.

(1) Assume $na_{ij}$ is even (and hence $m$ is odd by assumption) and $\ov{p}=0$.

We have
\begin{align*}
B^{(2)}_{i,\ov{0}} \Xi
&= \sum_{r+s=m}(-1)^r \qbinom{r+2}{2}_{i} B^{(r+2)}_{i,\ov{0}}  X B_{i,\ov{0}}^{(s)}
+ \sum_{\stackrel{r \text{ even}}{r+s=m} }(-1)^r \frac{[r]_{i}^2}{[2]_{i}} (q_i \tk_i) B^{(r)}_{i,\ov{0}}  X B_{i,\ov{0}}^{(s)}
\\
& \quad + \sum_{\stackrel{r \text{ odd}}{r+s=m} }(-1)^r \frac{[r+1]_{i}^2}{[2]_{i}} (q_i \tk_i) B^{(r)}_{i,\ov{0}}  X B_{i,\ov{0}}^{(s)}
\\%
&= \sum_{r+s=m+2}(-1)^r \qbinom{r}{2}_{i} B^{(r)}_{i,\ov{0}}  X B_{i,\ov{0}}^{(s)}
+ \sum_{\stackrel{r \text{ even}}{r+s=m} } (-1)^r \frac{[r]_{i}^2}{[2]_{i}} (q_i \tk_i) B^{(r)}_{i,\ov{0}}  X B_{i,\ov{0}}^{(s)}
\\
& \quad + \sum_{\stackrel{r \text{ odd}}{r+s=m} } (-1)^r \frac{[r+1]_{i}^2}{[2]_{i}} (q_i \tk_i) B^{(r)}_{i,\ov{0}}  X B_{i,\ov{0}}^{(s)},
\end{align*}
where  the last equation is obtained by a change of variables $r\mapsto r-2$ on the first summand.

On the other hand, we have
\begin{align*}
\Xi B^{(2)}_{i,\ov{0}}
&= \sum_{r+s=m}(-1)^r \qbinom{s+2}{2}_{i} B^{(r)}_{i,\ov{0}}  X B_{i,\ov{0}}^{(s+2)}
+ \sum_{\stackrel{s \text{ even}}{r+s=m} } (-1)^r \frac{[s]_{i}^2}{[2]_{i}} (q_i \tk_i) B^{(r)}_{i,\ov{0}}  X B_{i,\ov{0}}^{(s)}
\\
& \quad + \sum_{\stackrel{s \text{ odd}}{r+s=m} } (-1)^r \frac{[s+1]_{i}^2}{[2]_{i}} (q_i \tk_i) B^{(r)}_{i,\ov{0}}  X B_{i,\ov{0}}^{(s)}
\\%
&= \sum_{r+s=m+2} (-1)^r \qbinom{s}{2}_{i} B^{(r)}_{i,\ov{0}}  X B_{i,\ov{0}}^{(s)}
+ \sum_{\stackrel{r \text{ odd}}{r+s=m} } (-1)^r \frac{[s]_{i}^2}{[2]_{i}} (q_i \tk_i) B^{(r)}_{i,\ov{0}}  X B_{i,\ov{0}}^{(s)}
\\
& \quad + \sum_{\stackrel{r \text{ even}}{r+s=m} } (-1)^r \frac{[s+1]_{i}^2}{[2]_{i}} (q_i \tk_i) B^{(r)}_{i,\ov{0}}  X B_{i,\ov{0}}^{(s)},
\end{align*}
where  the last equation is obtained by a change of variables $s\mapsto s-2$ on the first summand; also note in the last 2 summands that $r$ is even if and only if $s$ is odd since $r+s=m$ is odd.

Similarly we have
\begin{align*}
B_{i} \Xi B_{i}
&= \sum_{r+s=m}(-1)^r [r+1]_{i} [s+1]_{i} B^{(r+1)}_{i,\ov{0}}  X B_{i,\ov{0}}^{(s+1)}
\\
& \quad + \sum_{\stackrel{r \text{ even}}{r+s=m} } (-1)^r [r]_{i} [s+1]_{i} (q_i \tk_i) B^{(r-1)}_{i,\ov{0}}  X B_{i,\ov{0}}^{(s+1)}
\\
& \quad + \sum_{\stackrel{r \text{ odd}}{r+s=m} } (-1)^r [r+1]_{i} [s]_{i} (q_i \tk_i) B^{(r+1)}_{i,\ov{0}}  X B_{i,\ov{0}}^{(s-1)}
\\%
&= \sum_{r+s=m+2} (-1)^{r-1} [r]_{i} [s]_{i} B^{(r)}_{i,\ov{0}}  X B_{i,\ov{0}}^{(s)}
 + \sum_{\stackrel{r \text{ odd}}{r+s=m} } (-1)^{r-1} [r+1]_{i} [s]_{i} (q_i \tk_i) B^{(r)}_{i,\ov{0}}  X B_{i,\ov{0}}^{(s)}
\\
& \quad + \sum_{\stackrel{r \text{ even}}{r+s=m} } (-1)^{r-1} [r]_{i} [s+1]_{i} (q_i \tk_i) B^{(r)}_{i,\ov{0}}  X B_{i,\ov{0}}^{(s)},
\end{align*}
where the last equation is obtained by changes of variables $(r\mapsto r-1, s\mapsto s-1)$,  $(r\mapsto r+1, s\mapsto s-1)$, (and respectively, $(r\mapsto r-1, s\mapsto s+1)$) on the first, second, (and respectively, third) summands.

Collecting the 3 identities for $B^{(2)}_{i,\ov{0}} \Xi, \Xi B^{(2)}_{i,\ov{0}}, B_{i} \Xi B_{i}$ above, we now compute the sum $S$ defined in \eqref{eq:SZ}. The coefficient of $B^{(r)}_{i,\ov{0}}  X B_{i,\ov{0}}^{(s)}$ in $S$, where $r+s=m+2$, is equal to
\begin{align*}
(-1)^r \qbinom{r}{2}_{i}
+ (-1)^r \qbinom{s}{2}_{i}
- \frac{q_i^{m+1} +q_i^{-m-1}}{[2]_{i}} \cdot (-1)^{r-1} [r]_{i} [s]_{i} =(-1)^r \qbinom{m+2}{2}_{i}.
\end{align*}
The coefficient of $(q_i \tk_i) B^{(r)}_{i,\ov{0}}  X B_{i,\ov{0}}^{(s)}$ in $S$, for $r+s=m$ and $r$ even, is equal to
\begin{align*}
(-1)^r \frac{[r]_{i}^2}{[2]_{i}}
+ (-1)^r \frac{[s+1]_{i}^2}{[2]_{i}}
- \frac{q_i^{m+1} +q_i^{-m-1}}{[2]_{i}} \cdot (-1)^{r-1} [r]_{i} [s+1]_{i}
=(-1)^r \frac{[m+1]^2_{i}}{[2]_{i}}.
\end{align*}
Similarly, the coefficient of $(q_i \tk_i) B^{(r)}_{i,\ov{0}}  X B_{i,\ov{0}}^{(s)}$ in $S$, for $r+s=m$ and $r$ odd, is equal to
\begin{align*}
(-1)^r \frac{[r+1]_{i}^2}{[2]_{i}}
+ (-1)^r \frac{[s]_{i}^2}{[2]_{i}}
- \frac{q_i^{m+1} +q_i^{-m-1}}{[2]_{i}} \cdot (-1)^{r-1} [r+1]_{i} [s]_{i}
=(-1)^r \frac{[m+1]^2_{i}}{[2]_{i}}.
\end{align*}
Summarizing, we have obtained
\begin{align*}
 B^{(2)}_{i,\ov{0}} \Xi & + \Xi B^{(2)}_{i,\ov{0}+\ov{na_{ij}}} - \frac{q_i^{m+1} +q_i^{-m-1}}{[2]_{i}} B_{i} \Xi B_{i}
\\
& =  \qbinom{m+2}{2}_{i} \sum_{r+s=m+2} (-1)^r B^{(r)}_{i,\ov{0}}  X B_{i,\ov{0}}^{(s)}
+ \frac{[m+1]^2_{i}}{[2]_{i}} \Xi.
\end{align*}
That is, we have established the identity \eqref{eq:SZZ} with $\xi =\frac{[m+1]^2_{i}}{[2]_{i}}$.

The proofs for the identity \eqref{eq:SZZ} in the remaining Cases (2)--(4) below are similar, and we will only write down the main formulas.

\vspace{2mm}

(2) Assume $na_{ij}$ is even (and hence $m$ is odd) and $\ov{p}=1$.
We have
\begin{align*}
B^{(2)}_{i,\ov{1}} \Xi
&= \sum_{r+s=m+2}(-1)^r \qbinom{r}{2}_{i} B^{(r)}_{i,\ov{1}}  X B_{i,\ov{1}}^{(s)}
+ \sum_{\stackrel{r \text{ even}}{r+s=m} } (-1)^r \frac{[r+1]_{i}^2-1}{[2]_{i}} (q_i \tk_i) B^{(r)}_{i,\ov{1}}  X B_{i,\ov{1}}^{(s)}
\\
& \quad + \sum_{\stackrel{r \text{ odd}}{r+s=m} } (-1)^r \frac{[r]_{i}^2-1}{[2]_{i}} (q_i \tk_i) B^{(r)}_{i,\ov{1}}  X B_{i,\ov{1}}^{(s)};
\end{align*}
\begin{align*}
\Xi B^{(2)}_{i,\ov{1}}
&= \sum_{r+s=m+2} (-1)^r \qbinom{s}{2}_{i} B^{(r)}_{i,\ov{1}}  X B_{i,\ov{1}}^{(s)}
+ \sum_{\stackrel{r \text{ odd}}{r+s=m} } (-1)^r \frac{[s+1]_{i}^2-1}{[2]_{i}} (q_i \tk_i) B^{(r)}_{i,\ov{1}}  X B_{i,\ov{1}}^{(s)}
\\
& \quad + \sum_{\stackrel{r \text{ even}}{r+s=m} } (-1)^r \frac{[s]_{i}^2-1}{[2]_{i}} (q_i \tk_i) B^{(r)}_{i,\ov{1}}  X B_{i,\ov{1}}^{(s)};
\end{align*}
\begin{align*}
B_{i} \Xi B_{i}
&= \sum_{r+s=m+2} (-1)^{r-1} [r]_{i} [s]_{i} B^{(r)}_{i,\ov{1}}  X B_{i,\ov{1}}^{(s)}
 + \sum_{\stackrel{r \text{ even}}{r+s=m} } (-1)^{r-1} [r+1]_{i} [s]_{i} (q_i \tk_i) B^{(r)}_{i,\ov{1}}  X B_{i,\ov{1}}^{(s)}
\\
& \quad + \sum_{\stackrel{r \text{ odd}}{r+s=m} } (-1)^{r-1} [r]_{i} [s+1]_{i} (q_i \tk_i) B^{(r)}_{i,\ov{1}}  X B_{i,\ov{1}}^{(s)}.
\end{align*}
From these formulas, we see that the coefficient of $B^{(r)}_{i,\ov{1}}  X B_{i,\ov{1}}^{(s)}$ (for $r+s=m+2$) in the sum $S$ defined in \eqref{eq:SZ}  is again equal to $(-1)^r \qbinom{m+2}{2}_{i}$;
the coefficient of $(q_i \tk_i) B^{(r)}_{i,\ov{1}}  X B_{i,\ov{1}}^{(s)}$ in $S$, for $r+s=m$ and $r$ even, is
\begin{align*}
&= (-1)^r \frac{[r+1]_{i}^2-1}{[2]_{i}}
+ (-1)^r \frac{[s]_{i}^2-1}{[2]_{i}}
- \frac{q_i^{m+1} +q_i^{-m-1}}{[2]_{i}}   (-1)^{r-1} [r+1]_{i} [s]_{i}
\\
&=(-1)^r \frac{[m+1]^2_{i}-2}{[2]_{i}};
\end{align*}
and the coefficient of $(q_i \tk_i) B^{(r)}_{i,\ov{1}}  X B_{i,\ov{1}}^{(s)}$ in $S$, for $r+s=m$ and $r$ odd, is
\begin{align*}
&= (-1)^r \frac{[r]_{i}^2-1}{[2]_{i}}
+ (-1)^r \frac{[s+1]_{i}^2-1}{[2]_{i}}
- \frac{q_i^{m+1} +q_i^{-m-1}}{[2]_{i}}   (-1)^{r-1} [r]_{i} [s+1]_{i}
\\
& =(-1)^r \frac{[m+1]^2_{i} -2}{[2]_{i}}.
\end{align*}
In this way, we have established the identity \eqref{eq:SZZ} with $\xi =\frac{[m+1]^2_{i}-2}{[2]_{i}}$.

\vspace{2mm}

(3) Assume $na_{ij}$ is odd (and hence $m$ is even) and $\ov{p}=1$.
We have
\begin{align*}
B^{(2)}_{i,\ov{1}} \Xi
&= \sum_{r+s=m+2}(-1)^r \qbinom{r}{2}_{i} B^{(r)}_{i,\ov{1}}  X B_{i,\ov{0}}^{(s)}
+ \sum_{\stackrel{r \text{ even}}{r+s=m} } (-1)^r \frac{[r+1]_{i}^2-1}{[2]_{i}} (q_i \tk_i) B^{(r)}_{i,\ov{1}}  X B_{i,\ov{0}}^{(s)}
\\
& \quad + \sum_{\stackrel{r \text{ odd}}{r+s=m} } (-1)^r \frac{[r]_{i}^2-1}{[2]_{i}} (q_i \tk_i) B^{(r)}_{i,\ov{1}}  X B_{i,\ov{0}}^{(s)};
\end{align*}
\begin{align*}
\Xi B^{(2)}_{i,\ov{0}}
&= \sum_{r+s=m+2} (-1)^r \qbinom{s}{2}_{i} B^{(r)}_{i,\ov{1}}  X B_{i,\ov{0}}^{(s)}
+ \sum_{\stackrel{r \text{ even}}{r+s=m} } (-1)^r \frac{[s]_{i}^2}{[2]_{i}} (q_i \tk_i) B^{(r)}_{i,\ov{1}}  X B_{i,\ov{0}}^{(s)}
\\
& \quad + \sum_{\stackrel{r \text{ odd}}{r+s=m} } (-1)^r \frac{[s+1]_{i}^2}{[2]_{i}} (q_i \tk_i) B^{(r)}_{i,\ov{1}}  X B_{i,\ov{0}}^{(s)};
\end{align*}
\begin{align*}
B_{i} \Xi B_{i}
&= \sum_{r+s=m+2} (-1)^{r-1} [r]_{i} [s]_{i} B^{(r)}_{i,\ov{1}}  X B_{i,\ov{0}}^{(s)}
 + \sum_{\stackrel{r \text{ even}}{r+s=m} } (-1)^{r-1} [r+1]_{i} [s]_{i} (q_i \tk_i) B^{(r)}_{i,\ov{1}}  X B_{i,\ov{0}}^{(s)}
\\
& \quad + \sum_{\stackrel{r \text{ odd}}{r+s=m} } (-1)^{r-1}  [r]_{i} [s+1]_{i}
(q_i \tk_i) B^{(r)}_{i,\ov{1}}  X B_{i,\ov{0}}^{(s)}.
\end{align*}
From these formulas, we see that the coefficient of $B^{(r)}_{i,\ov{1}}  X B_{i,\ov{0}}^{(s)}$ (for $r+s=m+2$) in the sum $S$ defined in \eqref{eq:SZ}  is again equal to $(-1)^r \qbinom{m+2}{2}_{i}$;
the coefficient of $(q_i \tk_i) B^{(r)}_{i,\ov{1}}  X B_{i,\ov{1}}^{(s)}$ in $S$, for $r+s=m$ and $r$ even, is equal to
\begin{align*}
(-1)^r \frac{[r+1]_{i}^2-1}{[2]_{i}}
+ (-1)^r \frac{[s]_{i}^2}{[2]_{i}}
- \frac{q_i^{m+1} +q_i^{-m-1}}{[2]_{i}}   (-1)^{r-1} [r+1]_{i} [s]_{i}
=(-1)^r \frac{[m+1]^2_{i}-1}{[2]_{i}};
\end{align*}
and the coefficient of $(q_i \tk_i) B^{(r)}_{i,\ov{1}}  X B_{i,\ov{1}}^{(s)}$ in $S$, for $r+s=m$ and $r$ odd, is equal to
\begin{align*}
(-1)^r \frac{[r]_{i}^2-1}{[2]_{i}}
+ (-1)^r \frac{[s+1]_{i}^2}{[2]_{i}}
- \frac{q_i^{m+1} +q_i^{-m-1}}{[2]_{i}}   (-1)^{r-1} [r]_{i} [s+1]_{i}
=(-1)^r \frac{[m+1]^2_{i} -1}{[2]_{i}}.
\end{align*}
Thus we have established the identity \eqref{eq:SZZ} with $\xi =\frac{[m+1]^2_{i}-1}{[2]_{i}}$.

\vspace{2mm}

(4) Assume $na_{ij}$ is odd (and hence $m$ is even) and $\ov{p}=0$. This case (with $\xi =\frac{[m+1]^2_{i}-1}{[2]_{i}}$) is completely parallel to Case (3), and it can also be obtained from (3) by applying a suitable anti-involution.

This completes the proof of the identity \eqref{eq:SZZ} and hence Proposition~\ref{prop:SLinduct}.
\end{proof}

\subsection{Non-standard Serre-Lusztig in $\U$}

We obtain some curious non-standard Serre-Lusztig relations for $\U$, which is a counterpart of Theorem~{\bf B} for $\tUi$.

\begin{cor}
The following identities hold in $\U^+$, for any $i \neq j\in \I$ and $n,t \in \Z_{\ge 0}$:
\begin{align}
  \label{eq:nstdSL}
\sum_{r+s=1-na_{ij}+2t} (-1)^r E^{(r)}_{i} E_j^{(n)} E_{i}^{(s)} &=0.
\end{align}
\end{cor}

\begin{proof}
A $\U$-version of Proposition~\ref{prop:SLinduct} holds (by a similar and even simpler proof), where the $\imath$divided powers are replaced by the divided powers of $E_i$ in $\U^+$. Then the corollary follows from this variant of Proposition~\ref{prop:SLinduct} and the Serre-Lusztig relation \eqref{eq:minSLQG}.

Here is a second proof. A split $\imath$quantum group $\tUi$ with $\tk_i=\tk_j=0$ is isomorphic to $\U^-$. Theorem~{\bf B} then reduces to an $F$-counterpart of \eqref{eq:nstdSL}, which is equivalent to \eqref{eq:nstdSL}.
\end{proof}

\section{Serre-Lusztig relations of minimal degree, I}
   \label{sec:minimalSL}

Throughout this section, we assume $\tau j = \wb j =j \in \Iw$. We shall prove the following theorem (which is Theorem~{\bf A} for $\tau j=j= \wb j$).

\begin{thm}
\label{thm:minLS}
For any $i \neq j\in \Iw$ such that $\tau i=i= \wb i$ and $\tau j=j= \wb j$, the following identities hold in $\tUi$:
\begin{align}
  \label{eq:SerreBn1}
\sum_{r+s=1-na_{ij}} (-1)^r B^{(r)}_{i,\ov{p}} B_j^n B_{i,\ov{p}+\ov{na_{ij}}}^{(s)}=0, \quad (n\geq 0).
\end{align}
\end{thm}

\subsection{First reductions}

\begin{lem}
  \label{lem:span}
Suppose  $\tau j = \wb j =j \in \Iw$. For each $n\ge 0$, $B_j^n$ lies in the $\K(q)$-span of
\begin{align*}
\{E_j^\mu(\tK'_j)^{\mu+2k}F_j^\nu (\tK_j\tK'_j)^{\beta-k}   \mid n=\mu+\nu+2\beta, 0\leq k\leq \beta\}.
\end{align*}
\end{lem}
\begin{proof}
Follows by a simple induction on $n$ and using the definition $B_j =F_j +E_j \tK'_j$; note that $\tK_j\tK'_j$ is central in $\tU$.
\end{proof}

\begin{thm}
\label{thm:evenjfixed}
Suppose  $\tau j = \wb j =j \in \Iw$. Let $\mu, \nu, \beta \in\Z_{\geq0}$ and denote $n=\mu+\nu+2\beta$. The following (equivalent) identities hold:
\begin{align}
\sum_{r+s=1- na_{ij}} (-1)^r B^{(r)}_{i,\ov{p}} E_j^\mu K_j^{-(\mu+2\beta)}F_j^\nu B_{i,\ov{p}+\ov{na_{ij}}}^{(s)}& =0 \quad \in \U,
\label{eq:UiS1}
\\
\sum_{r+s=1- na_{ij}} (-1)^r B^{(r)}_{i,\ov{p}} E_j^\mu(\tK'_j)^{\mu+2\beta}F_j^\nu B_{i,\ov{p}+\ov{na_{ij}}}^{(s)}& =0 \quad \in \tU.
\label{eq:UiS2}
\end{align}
\end{thm}

\begin{proof}
By the same type of arguments in Proposition~\ref{prop:iSerre}, the 2 identities \eqref{eq:UiS1}--\eqref{eq:UiS2} are equivalent. We shall prove \eqref{eq:UiS1}.

The proof of \eqref{eq:UiS1} is very long and computational; it will occupy \S\ref{subsec:reduction}--\ref{subsec:evev} and Appendix~\ref{App:A}. In \S\ref{subsec:reduction}, the proof of \eqref{eq:UiS1} is reduced to the verification of 4 identities \eqref{eq:serre11F}--\eqref{eq:serre11evenodd}. The proof of \eqref{eq:serre11F} is given in \S\ref{subsec:proof1}--\ref{subsec:evev}, while similar proofs of \eqref{eq:serre11odd}--\eqref{eq:serre11evenodd} are outlined in Appendix~\ref{App:A}.
\end{proof}

We can now complete the proof of Theorem~\ref{thm:minLS} using Theorem~\ref{thm:evenjfixed} and Proposition~\ref{prop:SLinduct}.

\begin{proof} [Proof of Theorem~\ref{thm:minLS}]
Let $\mu, \nu, \beta \in\Z_{\geq0}$ such that $n=\mu+\nu+2\beta$, and let $0 \le k \le \beta$.
By Theorem~\ref{thm:evenjfixed} (with $\beta$ replaced by $k$) and noting that $\ov{na_{ij}} = \ov{(\mu+\nu+2k)a_{ij}}$, we have
\begin{align}
 \label{eq:SLk}
\sum_{r+s=1- (\mu+\nu+2k)a_{ij}} (-1)^r B^{(r)}_{i,\ov{p}} E_j^\mu(\tK'_j)^{\mu+2k}F_j^\nu B_{i,\ov{p}+\ov{na_{ij}}}^{(s)}=0.
\end{align}
By Proposition~\ref{prop:SLinduct}, we have by induction on $t\in \Z_{\ge 0}$ that
\begin{align*}
\sum_{r+s=1- (\mu+\nu+2k)a_{ij} +2t} (-1)^r B^{(r)}_{i,\ov{p}} E_j^\mu(\tK'_j)^{\mu+2k}F_j^\nu B_{i,\ov{p}+\ov{na_{ij}}}^{(s)}=0;
\end{align*}
note the identity \eqref{eq:SLk} serves as the base case for the induction with $m =1- (\mu+\nu+2k)a_{ij}$ in Proposition~\ref{prop:SLinduct}.
In particular, for $t =(k -\beta) a_{ij}$, the above identity leads to the following identity (where a power of the central element $\tK_j\tK'_j$ in $\tU$ is freely inserted):
\begin{align*}
\sum_{r+s=1- (\mu+\nu+2\beta)a_{ij}} (-1)^r B^{(r)}_{i,\ov{p}} E_j^\mu(\tK'_j)^{\mu+2k}F_j^\nu (\tK_j\tK'_j)^{\beta-k} B_{i,\ov{p}+\ov{na_{ij}}}^{(s)}=0.
\end{align*}
The theorem follows from this identity and Lemma~\ref{lem:span}.
\end{proof}

\begin{rem}
 \label{rem:BVT}
Conjectures and examples of Serre-Lusztig relations of minimal degree were proposed earlier by Baseilhac and Vu \cite{BaV14, BaV15} for certain split $\imath$quantum groups (with $a_{ij}=-1, -2$); the conjectural relations on $q$-Onsager algebra in \cite{BaV14} (with $a_{ij}=-2$) were subsequently established in \cite{Ter18} by applying the braid group action from \cite{BaK20}. Their formulas are expressed in terms of monomials in Chevalley generators $B_i$ and look very different from the compact formula given in Theorem~\ref{thm:minLS}. We verified that our formulas for $n=2, 3, 4, 5$ (with $a_{ij}=-1$) and for $n=2,3$ (with $a_{ij}=-2$) agree with theirs. An anonymous referee has kindly provided further evidence showing that the combinatorics used in these two different formulations of Serre-Lusztig relations of minimal degree is agreeable to each other, for all $n$.

It will be interesting to understand better possible connections between the aforementioned works using tridiagonal pairs and ours using $\imath$divided powers.
Our current proof of Theorem \ref{thm:minLS} (and then of Theorem {\bf A}) is long and computational.  It will be interesting to see if the approach of \cite{Ter18} for the $q$-Onsager algebra can be extended to arbitrary (quasi-split) $\imath$quantum groups and if this will lead to an alternative proof of Theorem \ref{thm:minLS}.
\end{rem}

The following simple lemma will be used later.
\begin{lem}
  \label{lem:span0}
Suppose  $\tau j=j= \wb j$.
  \begin{enumerate}
\item
$B_{j,\ov{t}}^{(n)}$ is a linear combination of $\{ B_j^{n-2t} \tk_j^t \mid 0\le t \le \lfloor n/2 \rfloor \}$;
\item
$B_{j}^{n}$ is a linear combination of $\{ B_{j,\ov{t}}^{(n-2t)}  \tk_j^t \mid 0\le t \le \lfloor n/2 \rfloor \}$.
\end{enumerate}
\end{lem}

\begin{proof}
Part (1) follows from the definition of $B_{j,\ov{t}}^{(n)}$ \eqref{eq:iDPodd}--\eqref{eq:iDPev}.
Then, $\{ B_j^{n-2t} \tk_j^t \mid 0\le t \le \lfloor n/2 \rfloor \}$ and $\{ B_{j,\ov{t}}^{(n-2t)}  \tk_j^t \mid 0\le t \le \lfloor n/2 \rfloor \}$ span the same vector space $V_n$ and they are bases of $V_n$. Now Part (2) follows.
\end{proof}

\subsection{Reduction of Theorem~\ref{thm:evenjfixed} to \eqref{eq:serre11F}--\eqref{eq:serre11evenodd}}
  \label{subsec:reduction}

{\em For the proof of \eqref{eq:UiS1} in Theorem~\ref{thm:evenjfixed} in \S\ref{subsec:reduction}--\ref{subsec:evev} and Appendix~\ref{App:A}, we set
\[
i=1 \in \Iw, \quad j=2 \in \Iw.
\]
}

We recall the following PBW expansion formulas of the $\imath$divided powers.

\begin{lem}
 [\text{\cite[Propositions 2.8, 3.5]{BeW18}}]
   \label{lem:iDPdot}
For $m\ge 1$ and $\la \in \Z$, we have
\begin{align}
B_{1,\ev}^{(2m)} \onestar_{2\la}
&\small
= \sum_{c=0}^m \sum_{a=0}^{2m-2c} q_1^{2(a+c)(m-a-\la)-2ac-\binom{2c+1}{2}} \qbinom{m-c-a-\la}{c}_{q_1^2}
 E^{(a)}_1  F^{(2m-2c-a)}_1\onestar_{2\la},
\label{t2mdot}
\\
B_{1,\ev}^{(2m-1)} \onestar_{2\la}
&=  \sum_{c=0}^{m-1} \sum_{a=0}^{2m-1-2c}
q_1^{2(a+c)(m-a-\la)-2ac-a-\binom{2c+1}{2}} \times
 \label{t2m-1dot}   \\
& \qquad\qquad\qquad \qbinom{m-c-a-\la-1}{c}_{q_1^2}  E_1^{(a)}  F_1^{(2m-1-2c-a)}\onestar_{2\la},
\notag
\\
B_{1,\odd}^{(2m)} \onestar_{2\la-1}
&= \sum_{c=0}^m \sum_{a=0}^{2m-2c}
 q_1^{2(a+c)(m-a-\la)-2ac+a-\binom{2c}{2}} \times
 \label{t2mdot2}  \\
& \qquad\qquad\qquad \qbinom{m-c-a-\la}{c}_{q_1^2}
 E^{(a)}_1  F^{(2m-2c-a)}_1\onestar_{2\la-1},
\notag
\\
B_{1,\odd}^{(2m+1)} \onestar_{2\la-1}
&=  \sum_{c=0}^{m} \sum_{a=0}^{2m+1-2c}
q_1^{2(a+c)(m-a-\la)-2ac+2a-\binom{2c}{2}} \times
\label{t2m-1dot2} \\
& \qquad\qquad\qquad  \qbinom{m-c-a-\la+1}{c}_{q_1^2}  E_1^{(a)}  F_1^{(2m+1-2c-a)}\onestar_{2\la-1}.
\notag
\end{align}
\end{lem}

The necessity of applying different formulas in Lemma~\ref{lem:iDPdot} forces us to divide the proof of \eqref{eq:UiS1} in Theorem~\ref{thm:evenjfixed} into the 4 cases \eqref{eq:serre11F}--\eqref{eq:serre11evenodd}.

\begin{align}
  \label{eq:serre11F}
  \sum_{r=0}^{1-a_{12}(\mu+\nu+2\beta)} (-1)^r  B_{1,\ev}^{(r)} E_2^\mu K_2^{-(\mu+2\beta)} F_2^\nu B_{1,\ev}^{(1 -a_{12}(\mu+\nu+2\beta)-r)}=0,
\text{ if } a_{12}(\mu+\nu)\in 2\Z_{\geq0};
 \\
  \label{eq:serre11odd}
  \sum_{r=0}^{1-a_{12}(\mu+\nu+2\beta)} (-1)^r  B_{1,\odd}^{(r)}E_2^\mu K_2^{-(\mu+2\beta)} F_2^\nu B_{1,\odd}^{(1 -a_{12}(\mu+\nu+2\beta)-r)}=0,
  \text{ if } a_{12}(\mu+\nu) \in 2\Z_{\geq0};
\\
  \label{eq:serre11oddeven}
 \sum_{r=0}^{1-a_{12}(\mu+\nu+2\beta)} (-1)^r  B_{1,\odd}^{(r)}E_2^\mu K_2^{-(\mu+2\beta)} F_2^\nu B_{1,\ev}^{(1 -a_{12}(\mu+\nu+2\beta)-r)}=0,
  \text{ if } a_{12}(\mu+\nu) \in2\Z_{\geq0}+1;
\\
\label{eq:serre11evenodd}
\sum_{r=0}^{1-a_{12}(\mu+\nu+2\beta)} (-1)^r  B_{1,\ev}^{(r)}E_2^\mu K_2^{-(\mu+2\beta)} F_2^\nu B_{1,\odd}^{(1 -a_{12}(\mu+\nu+2\beta)-r)}=0,
 \text{ if } a_{12}(\mu+\nu) \in 2\Z_{\geq0}+1.
\end{align}

In \S\ref{subsec:proof1}--\ref{subsec:evev} below, we shall prove the identity \eqref{eq:serre11F};  similar proofs of the other identities \eqref{eq:serre11odd}--\eqref{eq:serre11evenodd} are postponed to  Appendix \ref{App:A}.

\subsection{PBW expansion of LHS\eqref{eq:serre11F}}
  \label{subsec:proof1}

To prove \eqref{eq:serre11F}, we shall establish its counterpart in $\Udot$ as formulated in \eqref{eq:dot0} below, thanks to Remark \ref{rem:u=0}. In the remainder of this section, we denote
\[
\alpha=-a_{12}\in\Z_{\ge 0}.
\]

For any $\beta\in\Z_{\geq0}$, we shall use (\ref{t2mdot})--(\ref{t2m-1dot}) to rewrite the element
\begin{align}
  \label{eq:BFBjfix}
& \small\sum_{r=0}^{\alpha(\mu+\nu+2\beta)+1}
(-1)^rB_{1,\ev}^{(r)}E_2^\mu K_2^{-(\mu+2\beta)}F_2^\nu B_{1,\ev}^{(\alpha(\mu+\nu+2\beta)+1-r)}\onestar_{2\la}\in\Udot,
\end{align}
for any $\la\in\Z$, in terms of monomials in $E_1, F_1, F_2, E_2, \tK_2^{-1}$.

\vspace{2mm}

\noindent{\underline{Case I: $r$ is even}.} It follows from \eqref{t2m-1dot} that
\begin{align*}
B_{1,\ev}^{(\alpha(\mu+\nu+2\beta)+1-r)}\onestar_{2\la}
&=\sum_{c=0}^{\frac{\alpha}{2}(\mu+\nu+2\beta)-\frac{r}{2}} \sum_{a=0}^{\alpha(\mu+\nu+2\beta)+1-r-2c}q_1^{(a+c)(\alpha(\mu+\nu+2\beta)+2-r-2a-2\la)-2ac-a-c(2c+1)} \\&
\qquad \cdot \qbinom{\frac{\alpha}{2}(\mu+\nu+2\beta)-\frac{r}{2}-c-a-\la}{c}_{q_1^2}  E_1^{(a)}  F_1^{(\alpha(\mu+\nu+2\beta)+1-r-2c-a)}\onestar_{2\la}.
\notag
\end{align*}
By \eqref{eqn: idempotent Ei Fi}--\eqref{eqn: idempotent Ej Fj} we have $F_2\onestar_{\la}=\onestar_{\la+\alpha}F_2$,
$E_2\onestar_{\la}=\onestar_{\la-\alpha}E_2$, and hence
\begin{align*}
E_2^\mu & K_2^{-(\mu+2\beta)}F_2^\nu E_1^{(a)}  F_1^{(\alpha(\mu+\nu+2\beta)+1-r-2c-a)}\onestar_{2\la}
\\
=&\onestar_{2(\la+2a+r+2c-1-\frac{3\alpha\mu}{2}-\frac{\alpha\nu}{2}-2\alpha\beta)}E_2^\mu K_2^{-(\mu+2\beta)}F_2^\nu  E_1^{(a)}  F_1^{(\alpha(\mu+\nu+2\beta)+1-r-2c-a)}.
\end{align*}
Furthermore, by using (\ref{t2mdot}), we have
\begin{align*}
 B_{1,\ev}^{(r)} & \onestar_{2(\la+2a+r+2c-1-\frac{3\alpha\mu}{2}-\frac{\alpha\nu}{2}-2\alpha\beta)}
 \\
&= \sum_{e=0}^{\frac{r}{2}} \sum_{d=0}^{r-2e}
 q_1^{2(d+e)(\frac{3\alpha\mu}{2}+\frac{\alpha\nu}{2}+2\alpha\beta+1-d-\la-2a-\frac{r}{2}-2c)-2de-e(2e+1)}
\\
&\quad \cdot \qbinom{\frac{3\alpha\mu}{2}+\frac{\alpha\nu}{2}+2\alpha\beta+1-e-d-\la-2a-\frac{r}{2}-2c}{e}_{q_1^2}
\\
&\quad \cdot E_1^{(d)}  F_1^{(r-2e-d)}\onestar_{2(\la+2a+r+2c-1-\frac{3\alpha\mu}{2}-\frac{\alpha\nu}{2}-2\alpha\beta)}.
\end{align*}
Hence combining the above 3 computations gives us
\begin{align}
  \label{eq:BFBjfix}
&\quad \;B_{1,\ev}^{(r)}E_2^\mu K_2^{-(\mu+2\beta)} F_2^\nu B_{1,\ev}^{(\alpha(\mu+\nu+2\beta)+1-r)}\onestar_{2\la} \\
 &=\sum_{e=0}^{\frac{r}{2}} \sum_{d=0}^{r-2e}\sum_{c=0}^{\frac{\alpha}{2}(\mu+\nu+2\beta)-\frac{r}{2}} \sum_{a=0}^{\alpha(\mu+\nu+2\beta)+1-r-2c}\notag
\\
&\qquad q_1^{(a+c+d+e)(\alpha(\mu+\nu+2\beta)+1-r-2\la-2a-2c-2d-2e)+2\alpha(\mu+\beta)(d+e)+d}
  \notag \\
 & \quad \cdot \qbinom{\frac{\alpha}{2}(\mu+\nu+2\beta)-\frac{r}{2}-c-a-\la}{c}_{q_1^2}
\qbinom{\frac{3\alpha\mu}{2}+\frac{\alpha\nu}{2}+2\alpha\beta+1-e-d-\la-2a-\frac{r}{2}-2c}{e}_{q_1^2}
   \notag \\
 & \quad \cdot E_1^{(d)}  F_1^{(r-2e-d)}E_2^\mu K_2^{-(\mu+2\beta)}F_2^\nu E_1^{(a)}  F_1^{(\alpha(\mu+\nu+2\beta)+1-r-2c-a)}\onestar_{2\la}
 \notag
 \\
 &=\sum_{e=0}^{\frac{r}{2}} \sum_{d=0}^{r-2e}\sum_{c=0}^{\frac{\alpha}{2}(\mu+\nu+2\beta)-\frac{r}{2}} \sum_{a=0}^{\alpha(\mu+\nu+2\beta)+1-r-2c}
 \notag\\
&\qquad q_1^{(a+c+d+e)(\alpha(\mu+\nu)+1-r-2\la-2a-2c-2d-2e)+2\alpha(\mu+\beta)(d+e)+d}q_1^{\alpha(\mu+2\beta)(r-2e-d)}
  \notag \\
 & \quad \cdot \qbinom{\frac{\alpha}{2}(\mu+\nu+2\beta)-\frac{r}{2}-c-a-\la}{c}_{q_1^2}
\qbinom{\frac{3\alpha\mu}{2}+\frac{\alpha\nu}{2}+2\alpha\beta+1-e-d-\la-2a-\frac{r}{2}-2c}{e}_{q_1^2}
   \notag \\
 & \quad \cdot  E_1^{(d)}  E_2^\mu K_2^{-(\mu+2\beta)} F_1^{(r-2e-d)}E_1^{(a)}F_2^\nu  F_1^{(\alpha(\mu+\nu+2\beta)+1-r-2c-a)}\onestar_{2\la}.
 \notag
\end{align}
Here the second equality follows by using \eqref{eq:EK} and \eqref{Q4}.

Next, inspired by the PBW basis of $\U$, we move the divided powers of $E_1$ in the middle to the left.
Using \eqref{eqn: idempotent Ei Fi}--\eqref{eqn: idempotent Ej Fj} we have
\begin{align*}
F_2^\nu F_1^{(\alpha(\mu+\nu+2\beta)+1-r-2c-a)}\onestar_{2\la} 
 &=\onestar_{2(\la+r+2c+a-\alpha\mu-2\alpha\beta-\frac{\alpha\nu}{2}-1)}F_2^\nu  F_1^{(\alpha(\mu+\nu+2\beta)+1-r-2c-a)}.
\end{align*}
Using \eqref{eqn:commutate-idempotent3} we have
\begin{align*}
F_1^{(r-2e-d)} & E_1^{(a)} \onestar_{2(\la+r+2c+a-\alpha\mu-2\alpha\beta-\frac{\alpha\nu}{2}-1)}\\
&=\sum^{\min\{a,r-2e-d\}}_{b=0}\qbinom{r-2e-d-a-2(\la+r+2c+a-\alpha\mu-2\alpha\beta-\frac{\alpha\nu}{2}-1)}{b}_{q_1}\\
&\qquad\qquad\qquad\quad \cdot E_1^{(a-b)}F_1^{(r-2e-d-b)}\onestar_{2(\la+r+2c+a-\alpha\mu-2\alpha\beta-\frac{\alpha\nu}{2}-1)}\\
&=\sum^{\min\{a,r-2e-d\}}_{b=0}\qbinom{2\alpha\mu+\alpha\nu+4\alpha\beta+2-2e-d-3a-2\la-4c-r}{b}_{q_1} \\
&\qquad\qquad\qquad\quad \cdot E_1^{(a-b)} F_1^{(r-2e-d-b)}\onestar_{2(\la+r+2c+a-\alpha\mu-2\alpha\beta-\frac{\alpha\nu}{2}-1)}.
\end{align*}

Plugging these new formulas into \eqref{eq:BFBjfix}, we obtain
\begin{align}
\small
&  \sum_{r=0,2\mid r}^{\alpha(\mu+\nu+2\beta)+1}
    B_{1,\ev}^{(r)}E_2^\mu K_2^{-(\mu+2\beta)}F_2^\nu B_{1,\ev}^{(\alpha(\mu+\nu+2\beta)+1-r)}\onestar_{2\la}\label{eqn: first even j fix}\\\notag
&=\sum_{r=0,2\mid r}^{\alpha(\mu+\nu+2\beta)+1}\sum_{c=0}^{\frac{\alpha}{2}(\mu+\nu+2\beta)-\frac{r}{2}}\sum_{e=0}^{\frac{r}{2}} \sum_{a=0}^{\alpha(\mu+\nu+2\beta)+1-r-2c}\sum_{d=0}^{r-2e}\sum^{\min\{a,r-2e-d\}}_{b=0}
\\
&\qquad q_1^{(a+c+d+e)(\alpha(\mu+\nu+2\beta)+1-r-2\la-2a-2c-2d-2e)+2\alpha(\mu+\beta)(d+e)+d}q_1^{\alpha(\mu+2\beta)(r+a-b-2e-d)}\notag\\
&\quad\cdot
\qbinom{2\alpha\mu+\alpha\nu+4\alpha\beta+2-2e-d-3a-2\la-4c-r}{b}_{q_1}\notag
\\
&\quad \cdot \qbinom{\frac{3\alpha\mu}{2}+\frac{\alpha\nu}{2}+2\alpha\beta+1-e-d-\la-2a-\frac{r}{2}-2c}{e}_{q_1^2}
\qbinom{\frac{\alpha}{2}(\mu+\nu+2\beta)-\frac{r}{2}-c-a-\la}{c}_{q_1^2}\notag
\\
&\quad \cdot E_1^{(d)}E_2^\mu E_1^{(a-b)} K_2^{-(\mu+2\beta)}F_1^{(r-2e-d-b)}F_2^\nu F_1^{(\alpha(\mu+\nu+2\beta)+1-r-2c-a)}\onestar_{2\la}.
\notag
\end{align}

\vspace{2mm}
\noindent{\underline{Case II: $r$ is odd}.}
Similarly, by (\ref{t2mdot}) we have
\begin{align*}
&\quad B_{1,\ev}^{(\alpha(\mu+\nu+2\beta)+1-r)}\onestar_{2\la}
\\
&=\sum_{c=0}^{\frac{\alpha}{2}(\mu+\nu+2\beta)+\frac{1-r}{2}} \sum_{a=0}^{\alpha(\mu+\nu+2\beta)+1-r-2c} q_1^{(a+c)(\alpha(\mu+\nu+2\beta)+1-r-2a-2\la)-2ac-c(2c+1)} \\&
\qquad \cdot \qbinom{\frac{\alpha}{2}(\mu+\nu+2\beta)+\frac{1-r}{2}-c-a-\la}{c}_{q_1^2}  E_1^{(a)}  F_1^{(\alpha(\mu+\nu+2\beta)+1-r-2c-a)}\onestar_{2\la}.
\notag
\end{align*}
Using (\ref{t2m-1dot}) we have
\begin{align*}
B_{1,\ev}^{(r)} & \onestar_{2(\la+2a+r+2c-1-\frac{3\alpha\mu}{2}-\frac{\alpha\nu}{2}-2\alpha\beta)}\\
&=\sum_{e=0}^{\frac{r-1}{2}} \sum_{d=0}^{r-2e}q_1^{2(d+e)(\frac{3\alpha\mu}{2}+\frac{\alpha\nu}{2}+2\alpha\beta+1-d-\la-2a-\frac{r-1}{2}-2c)-2de-d-e(2e+1)} \\
&\quad \cdot \qbinom{\frac{3\alpha\mu}{2}+\frac{\alpha\nu}{2}+2\alpha\beta-e-d-\la-2a-\frac{r-1}{2}-2c}{e}_{q_1^2}  \\
&\quad \cdot E_1^{(d)}  F_1^{(r-2e-d)}\onestar_{2(\la+2a+r+2c-1-\frac{3\alpha\mu}{2}-\frac{\alpha\nu}{2}-2\alpha\beta)}.
\end{align*}
Combining the above two formulas and simplifying the resulting expression, we obtain the following equality:
\begin{align}
& \sum_{r=1,2\nmid r}^{\alpha(\mu+\nu+2\beta)+1} B_{1,\ev}^{(r)}E_2^\mu K_2^{-(\mu+2\beta)}F_2^\nu B_{1,\ev}^{(\alpha(\mu+\nu+2\beta)+1-r)}\onestar_{2\la}\label{eqn: first odd j fix}\\
&  =  \sum_{r=1,2\nmid r}^{\alpha(\mu+\nu+2\beta)+1}\sum_{c=0}^{\frac{\alpha}{2}(\mu+\nu+2\beta)+\frac{1-r}{2}} \sum_{e=0}^{\frac{r-1}{2}}
\sum_{a=0}^{\alpha(\mu+\nu+2\beta)+1-r-2c}\sum_{d=0}^{r-2e}\sum^{\min\{a,r-2e-d\}}_{b=0}
\notag\\
&\quad q_1^{(a+c+d+e)(\alpha(\mu+\nu+2\beta)+2-r-2\la-2a-2c-2d-2e)-a-2c+2\alpha(\mu+\beta)(d+e)}q_1^{\alpha(\mu+2\beta)(r+a-b-2e-d)} \notag \\
& \notag \quad \cdot
\qbinom{2\alpha\mu+\alpha\nu+4\alpha\beta+2-2e-d-3a-2\la-4c-r}{b}_{q_1}
\\
&\quad \cdot
\qbinom{\frac{3\alpha\mu}{2}+\frac{\alpha\nu}{2}+2\alpha\beta-e-d-\la-2a-\frac{r-1}{2}-2c}{e}_{q_1^2}   \notag
\qbinom{\frac{\alpha}{2}(\mu+\nu+2\beta)+\frac{1-r}{2}-c-a-\la}{c}_{q_1^2}\\\notag
 &\quad \cdot E_1^{(d)}E_2^\mu E_1^{(a-b)} K_2^{-(\mu+2\beta)}F_1^{(r-2e-d-b)}F_2^\nu F_1^{(\alpha(\mu+\nu+2\beta)+1-r-2c-a)}\onestar_{2\la}.
\notag
\end{align}

Therefore, by combining the computations \eqref{eqn: first even j fix}--\eqref{eqn: first odd j fix} which depend on the parity of $r$ above, we obtain the following formula for \eqref{eq:BFBjfix}:
\begin{align}
\label{eq:evev}
& \sum_{r=0}^{\alpha(\mu+\nu+2\beta)+1}(-1)^r
    B_{1,\ev}^{(r)}E_2^\mu K_2^{-(\mu+2\beta)}F_2^\nu B_{1,\ev}^{(\alpha(\mu+\nu+2\beta)+1-r)}\onestar_{2\la}  = \text{RHS}\, \eqref{eqn: first even j fix} - \text{RHS}\, \eqref{eqn: first odd j fix}.
\end{align}

We change variables by setting
\begin{align}
  \label{eq:change}
\begin{split}
l =a+d-b,
\quad
y &=r-2e-d-b,
\\
u =b +c +e,
\quad
w &=\alpha(\mu+\nu+2\beta) +2 -2\lambda -4u -2l   + d-y. 
\end{split}
\end{align}
(Sometimes, we shall write $w=w_y$ when it helps to indicate its dependence on $y$ below.)
Note
\[
r\equiv w-b \pmod 2.
\]
Define
\begin{align}
 \label{eq:T1jfix}
T(w, & u,l, \mu,\beta)
\\
:=& \small \sum_{\stackrel{b+c+e =u}{2|(w-b)}}
q_1^{(e-c-w)(\alpha\mu-l-u)+u(\alpha\mu-l)+ { 2\alpha\beta b+2\alpha\beta e}}\notag
\\
&\quad \cdot
\qbinom{\alpha(\mu+2\beta)+w-l}{b}_{q_1}
\qbinom{u-1 +\frac{w-b}2}{c}_{q_1^2}
\qbinom{\alpha(\mu+\beta) -l + \frac{w-b}{2}}{e}_{q_1^2}
\notag
\\
-&\small \sum_{\stackrel{b+c+e =u}{2 \nmid (w-b)} }
q_1^{(e-c-w)(\alpha\mu+1-l-u)+u(\alpha\mu-l)+w+ { 2\alpha\beta b+2\alpha\beta e}}\notag
\\
&\quad \cdot
\qbinom{\alpha(\mu+2\beta)+w-l}{b}_{q_1}
\qbinom{u-1 +\frac{w-b+1}2}{c}_{q_1^2}
\qbinom{\alpha(\mu+\beta) -l + \frac{w-b-1}{2}}{e}_{q_1^2}.
\notag
\end{align}

In these new notations, we rewrite \eqref{eq:evev} as
\begin{align}
\label{eq:evev3jfix}
 \sum_{r=0}^{\alpha (\mu+\nu+2\beta)+1}
& (-1)^r    B_{1,\ev}^{(r)}E_2^\mu K_2^{-(\mu+2\beta)}F_2^\nu B_{1,\ev}^{(\alpha(\mu+\nu+2\beta)+1-r)}\onestar_{2\la}\\\notag
& =\sum_{u=0}^{\frac{\alpha}{2}(\mu+\nu+2\beta)}\sum_{l=0}^{\alpha(\mu+\nu+2\beta)+1-2u} \sum_{d=0}^{l} \sum_{y=0}^{\alpha(\mu+\nu+2\beta)+1-l-2u}
\\
&\quad \cdot q_1^{d+ {2\alpha\beta(\alpha\mu+l+y)}+\alpha\mu(\alpha(\mu+\nu)+2-2\lambda -4u-l+2d)+(l +u)(u-2d-1)+l u}
T(w_y,u,l,\mu,\beta)
\notag\\
&\quad \cdot  E_1^{(d)}E_2^\mu E_1^{(l-d)} K_2^{-(\mu+2\beta)}F_1^{(y)}F_2^\nu F_1^{(\alpha(\mu+\nu+2\beta)+1-l-y-2u)}\onestar_{2\la}. \notag
\end{align}

\subsection{Final reduction to $q$-binomial identities}
 \label{subsec:ReductionGH}

Recall the function $G(w,u, \ell;p_0,p_1,p_2)$ from \eqref{eq:Twxyu}. It will be written below as $G_q(w,u, \ell;p_0,p_1,p_2)$ to indicate its dependence on $q$, since  in the application below it is necessary to replace $q$ by $q_1$.

\begin{lem}
For any $w\in\Z$, $u, l, \mu, \beta \in\Z_{\geq0}$,
we have
\begin{align}
T(w,u,l,\mu,\beta) &= (-1)^wq_1^{-(\alpha\mu-l-2u)w-u^2} \times
    \label{eq:gT=G}
\\
& \qquad\quad G_{q_1} (w,u,0 ;\alpha(\mu+2\beta)-l ,u-1,\alpha(\mu+\beta)-l ),
\notag
 \\
T(w+1,u,l,\mu,\beta) &= -q_1^{-(\alpha\mu-l-2u)}T(w,u,l,\mu,\beta).
 \label{eq:w+1}
\end{align}
\end{lem}

\begin{proof}
The identity \eqref{eq:gT=G} follows by definitions of $T$ and $G$ in \eqref{eq:T1jfix} and \eqref{eq:Twxyu}.

Using \eqref{eq:Gx+1w} and \eqref{eq:G=0}, we have
\begin{align*}
G(w,u,0;p_0,p_1,p_2)=G(w+1,u,0 ;p_0,p_1,p_2).
\end{align*}
This identity can then be converted into the identity \eqref{eq:w+1} with the help of \eqref{eq:gT=G}.
\end{proof}

\begin{prop}
\label{prop:gHserre}
The following identities hold:
\begin{align}
\label{eq:gHigherF}
 &\sum_{y=0}^{\alpha(\mu+\nu+2\beta)+1-l-2u}q_1^{d+ {2\alpha\beta(\alpha\mu+l+y)}+2\alpha\mu (\alpha(\mu+\nu)+2-2\lambda -4u-l+2d)+(l +u)(u-2d-1)+l u}\times\\
&\qquad T(w_y,u,l,\mu,\beta) E_1^{(d)}E_2^\mu E_1^{(l-d)} K_2^{ -(\mu+2\beta)} F_1^{(y)}F_2^\nu F_1^{(\alpha(\mu+\nu+2\beta)+1-l-y-2u)}\onestar_{2\la} =0, \notag
\end{align}
if $\alpha\mu\geq l+2u-2\alpha\beta$;
\begin{align}
\label{eq:gHigherE}
&\sum_{d=0}^{\ell} q_1^{d+ {2\alpha\beta(\alpha\mu+l+y)}+2\alpha\mu (\alpha(\mu+\nu)+2-2\lambda -4u-l+2d)+(l +u)(u-2d-1)+l u}\times\\
&\qquad T(w_y,u,l,\mu,\beta) E_1^{(d)}E_2^\mu E_1^{(l-d)} K_2^{ -(\mu+2\beta)} F_1^{(y)}F_2^\nu F_1^{(\alpha(\mu+\nu+2\beta)+1-l-y-2u)}\onestar_{2\la} =0, \notag
\end{align}
if $l> \alpha\mu$.
\end{prop}

\begin{proof}
We prove \eqref{eq:gHigherF}.
When comparing the coefficients $T(w_y,u,l,\mu,\beta)$ for various $y$, we keep in mind that $w_{y} = w_{y-1}-1$ by definition of $w$ in \eqref{eq:change} and hence $w_y=w_0 -y$.
Then by using \eqref{eq:w+1} and an induction on $y$ we obtain
\[
T(w_y,u,l,\mu,\beta) = (-1)^y q_1^{y(\alpha\mu-l-2u)}T(w_0,u,l,\mu,\beta).
\]
By a Serre-Lusztig relation, which is the $F$-analog of \eqref{eq:LS} with $n =\nu$ and $m=\alpha(\mu+\nu+2\beta)+1-l-2u$, we have
\begin{align*}
 &\sum_{y=0}^{\alpha(\mu+\nu+2\beta)+1-l-2u} q_1^{{2\alpha\beta y}} T(w_y,u,l,\mu,\beta) F_1^{(y)}F_2^\nu F_1^{(\alpha(\mu+\nu+2\beta)+1-l-y-2u)}\onestar_{2\la}
 \\
 &
 = T(w_0,u,l,\mu,\beta)
 \sum_{y=0}^{\alpha(\mu+\nu+2\beta)+1-l-2u}
 (-1)^y q_1^{y(2\alpha\beta +\alpha\mu-l-2u)} F_1^{(y)}F_2^\nu F_1^{(\alpha(\mu+\nu+2\beta)+1-l-y-2u)}\onestar_{2\la}
 \\
 &=0. \notag
\end{align*}
Therefore, we have
\begin{align*}
& \text{LHS}\, \eqref{eq:gHigherF} =
\\
 & \Big(q_1^{d+{2\alpha\beta(\alpha\mu+l)}+2\alpha\mu (\alpha(\mu+\nu)+2-2\lambda -4u-l+2d)+(l +u)(u-2d-1)+l u} E_1^{(d)}E_2^\mu E_1^{(l-d)} K_2^{ -(\mu+2\beta)}\Big)\\
 &\cdot \Big(\sum_{y=0}^{\alpha(\mu+\nu+2\beta)+1-l-2u}q_1^{2\alpha\beta y} T(w,u,l,\mu,\beta) F_1^{(y)}F_2^\nu F_1^{(\alpha(\mu+\nu+2\beta)+1-l-y-2u)}\onestar_{2\la}\Big)=0. \notag
\end{align*}

The proof of \eqref{eq:gHigherE} by using the Serre-Lusztig relation between $E_1, E_2$ is entirely similar, and hence will be skipped.
\end{proof}

\begin{prop}
\label{prop:gT=0}
For any $l\leq \alpha\mu \leq 2u+l-1-2\alpha\beta$, we have
$T(w,u,l,\mu,\beta)=0$.
\end{prop}

\begin{proof}
Recall that $T(w,u,l,\mu,\beta)$ is proportional to $G_{q_1}(w,u,0;\alpha(\mu+2\beta)-l,u-1,\alpha(\mu+\beta)-l)$; see \eqref{eq:gT=G}.

Using \eqref{eq:Gk}--\eqref{eq:Godd}, we see that if $2\mid (\alpha\mu-l)$, then
\begin{align}
 \label{eq:parity0}
&G_{q}(w,u,0;\alpha(\mu+2\beta)-l,u-1,\alpha(\mu+\beta)-l) \\
=&q^{2u(\alpha(\mu+2\beta)-l)}G_{q} \Big(w+\alpha(\mu+2\beta)-l,u,0;0,u-1-\frac{\alpha(\mu+2\beta)-l}{2},\frac{\alpha\mu-l}{2} \Big).
\notag
\end{align}
Similarly, if $2\nmid (\alpha\mu-l)$, then
\begin{align}
  \label{eq:parity1}
&G_{q}(w,u,0;\alpha(\mu+2\beta)-l,u-1,\alpha(\mu+\beta)-l)
\\
=&q^{2u(\alpha(\mu+2\beta)-l)}G_{q} \Big(w+\alpha(\mu+2\beta)-l,u,0;0,\frac{\alpha\mu-l-1}{2},u-1-\frac{\alpha(\mu+2\beta)-l-1}{2} \Big). \notag
\end{align}

We now proceed by separating into 2 cases, depending on the parity of $(\alpha\mu-l)$.
We shall give the details below when $2\mid(\alpha\mu-l)$ using \eqref{eq:parity0};
the other case is entirely similar using \eqref{eq:parity1} and will be skipped.

\vspace{2mm}

Assume $2\mid(\alpha\mu-l)$ from now on. Then by \eqref{eq:parity0} and Lemma~\ref{lem:GH00},
\begin{align}
G_{q}(w,u,0;&\alpha(\mu+2\beta)-l,u-1,\alpha(\mu+\beta)-l)
\label{eq:GH}
\\
=& H(u,u-1-\frac{\alpha(\mu+2\beta)-l}{2},\frac{\alpha\mu-l}{2})
\notag
\\
=&\sum_{\stackrel{c,e\geq0}{c+e=u}}q^{2c+2c(u-1- \frac{\alpha(\mu+2\beta)-l}{2})+2e(\frac{\alpha\mu-l}{2})}\qbinom{u-1-\frac{\alpha(\mu+2\beta)-l}{2}}{c}_{q^2}\qbinom{\frac{\alpha\mu-l}{2}}{e}_{q^2}.
\notag
\end{align}
Since by assumption $l\leq \alpha\mu \leq 2u+l-1-2\alpha\beta$, we have $u-1-\frac{\alpha(\mu+2\beta)-l}{2}\geq0$, and $\frac{\alpha\mu-l}{2}\geq0$. Note that
\[
\Big( u-1-\frac{\alpha(\mu+2\beta)-l}{2} \Big) + \frac{\alpha\mu-l}{2}= u-1-2\alpha\beta<u=c+e.
\]
Then one of the $q^2$-binomials in each summand of the RHS of \eqref{eq:GH} must vanish, and hence
$G_{q}(w,u,0;\alpha(\mu+2\beta)-l,u-1,\alpha(\mu+\beta)-l)=0$.

This implies by \eqref{eq:gT=G} that $T(w,u,l,\mu,\beta)=0$ if $l\leq \alpha\mu \leq 2u+l-1-2\alpha\beta$.
\end{proof}

\subsection{Completing the proof of the identity \eqref{eq:serre11F}}
\label{subsec:evev}

Combining \eqref{eq:evev3jfix} and Propositions~ \ref{prop:gHserre}--\ref{prop:gT=0}, we conclude that
\begin{align}
 \label{eq:dot0}
\sum_{r=0}^{1-a_{12}(\mu+\nu+2\beta)}
    (-1)^r B_{1,\ev}^{(r)}E_2^\mu K_2^{-(\mu+2\beta)}F_2^\nu B_{1,\ev}^{(1-a_{12}(\mu+\nu+2\beta)-r)}\onestar_{2\la} =0
\end{align}
for any $\la\in\Z$ and $\mu,\nu, \beta \in\Z_{\geq0}$ such that $\mu+\nu+2\beta=n$.
The identity \eqref{eq:serre11F} follows from this by Remark \ref{rem:u=0}.

\section{Serre-Lusztig relations of minimal degree, II}
   \label{sec:minimalSL2}

In this section, we shall establish the analogue of Theorem~\ref{thm:minLS} on the Serre-Lusztig relations of minimal degree in the remaining two cases for $j \in \Iw$, cf. \S\ref{subsec:main}, (ii)  $\tau j \neq j$,  and (iii) $\tau j = j \neq \wb j$. Then we complete the proofs of Theorems~{\bf A} and {\bf B}.

\subsection{The case when $\tau j\neq j \in \Iw$}

\begin{lem}
  \label{lem:span2}
Assume $\tau j\neq j \in \Iw$. For each $n\ge 0$, $B_j^n$ lies in the $\K(q)$-span of
\begin{align*}
\{E_{\tau j}^\mu(\tK'_{j})^{\mu} F_{j}^{\nu}\mid \mu+\nu=n\}.
\end{align*}
\end{lem}

\begin{proof}
Recall $B_j =F_j +E_{\tau j} \tK'_j$. Note $F_j  (E_{\tau j} \tK'_j) = q_j^{-2}(E_{\tau j} \tK'_j ) F_j.$ Hence $B_j^n$ lies in the $\K(q)$-linear span of $\{(E_{\tau j} \tK'_{j})^{\mu} F_{j}^{\nu}\mid \mu+\nu=n\}$. As $(E_{\tau j} \tK'_{j})^{\mu} \in q^\Z E_{\tau j}^\mu(\tK'_{j})^{\mu}$, the lemma follows.
\end{proof}

\begin{prop}
\label{prop:jnotfixed}
Let $i, j\in \Iw$ such that $\tau i=i= \wb i$ and $\tau j \neq j$. Then, for any $n>0$ and $\ov{p} \in \Z_2$, we have
\begin{align*}
\sum_{r +s =1- na_{ij}} (-1)^r B^{(r)}_{i,\ov{p}} B_j^n B_{i,\ov{p}+\ov{na_{ij}}}^{(s)}=0.
\end{align*}
\end{prop}

\begin{proof}
It suffices to show that $\sum_{r +s =1- na_{ij}} (-1)^r B^{(r)}_{i,\ov{p}} E_{\tau j}^\mu(\tK'_{j})^{\mu}F_{j}^{\nu} B_{i,\ov{p}+\ov{na_{ij}}}^{(s)}=0$ in $\tU$ thanks to Lemma~\ref{lem:span2}, which is equivalent to showing
\begin{align}
  \label{eq:Serrej2}
\sum_{r +s =1- na_{ij}} (-1)^r B^{(r)}_{i,\ov{p}} E_{\tau j}^\mu K_{j}^{-\mu} F_{j}^{\nu} B_{i,\ov{p}+\ov{na_{ij}}}^{(s)}=0 \in \U.
\end{align}
Note that $a_{ij}=a_{i,\tau j}$. The proof of \eqref{eq:Serrej2} is exactly the same as for the special case of \eqref{eq:UiS1} in Theorem~\ref{thm:evenjfixed} with $\beta=0$, and hence omitted here.
\end{proof}

\subsection{The case when $\tau j= j \neq \wb j$}

For $j\in \Iw$, define
\begin{align}  \label{eq:Zj}
Z_j &= \frac{1}{q_j^{-1} -q_j}  \  r_j \big(\T_{w_{\bullet}} (E_j) \big) \in \U_{\Iblack}^+,
\end{align}
where $r_j$ is as in \eqref{eq:rr}. Recall that $B_j =F_j + \T_{\wb}(E_j) \tK'_j$.

The following are variants of \cite[Lemma 5.4]{BW18c} and \cite[Lemma 5.5]{BW18c} in the setting of $\tU$ and $\tUi$, where the scalar $\vs_j$ is replaced by the central element $\tK_j \tK'_j$.

\begin{lem}
  \label{lem:BWc}
We have
\begin{enumerate}
\item
$[F_j,   \T_{w_{\bullet}} (E_j)] = Z_j \tK_j$;
\item
$Z_j$ commutes with $F_j$,  $\T_{w_{\bullet}} (E_j)  \tK'_j$, and $B_j$;
\item
$F_j (\T_{w_{\bullet}} (E_j)  \tK'_j) -q_j^{-2} (\T_{w_{\bullet}} (E_j)  \tK'_j) F_j  = Z_j \tK_j \tK'_j.$
\end{enumerate}
\end{lem}

The following lemma follows readily from Lemma~\ref{lem:BWc}(2)(3).
\begin{lem}
  \label{lem:span3}
For each $n\ge 0$ and $j\in \Iw$, $B_j^n$ lies in the $\K(q)$-span of
\begin{align*}
\{ (\T_{\wb}E_j)^\mu (\tK'_j)^{\mu} F_j^\nu (Z_j \tK_j \tK'_j)^{\beta}   \mid n=\mu+\nu+2\beta \}.
\end{align*}
\end{lem}

\begin{prop}
\label{prop:jfixed2}
Let $i, j\in \Iw$ such that $\tau i=i= \wb i$ and $\tau j = j \neq \wb j$. Then, for any $n>0$ and $\ov{p} \in \Z_2$, we have
\begin{align*}
\sum_{r +s =1- na_{ij}} (-1)^r B^{(r)}_{i,\ov{p}} B_j^n B_{i,\ov{p}+\ov{na_{ij}}}^{(s)}=0.
\end{align*}
\end{prop}

\begin{proof}
Note that $Z_j \tK_j \tK'_j$ commutes with $\imath$divided powers of $B_i$.
Hence, by Lemma~\ref{lem:span3}, it suffices to show that $\sum_{r +s =1- na_{ij}} (-1)^r B^{(r)}_{i,\ov{p}} (\T_{\wb}E_j)^\mu(\tK'_{j})^{\mu}F_{j}^{\nu} B_{i,\ov{p}+\ov{na_{ij}}}^{(s)}=0$ in $\tU$, for $\mu, \nu$ such that $n-\mu -\nu \in 2\Z_{\ge 0}$, which is equivalent to showing
\begin{align}
  \label{eq:Serrej3}
\sum_{r +s =1- na_{ij}} (-1)^r B^{(r)}_{i,\ov{p}} (\T_{\wb}E_j)^\mu K_{j}^{-\mu}F_{j}^{\nu} B_{i,\ov{p}+\ov{na_{ij}}}^{(s)}=0 \in \U.
\end{align}
By Proposition~\ref{prop:SLinduct}, this further reduces to establishing the identity \eqref{eq:Serrej3} for $\mu+\nu=n$.

The proof of the identity \eqref{eq:Serrej3} for $\mu+\nu=n$ is essentially the same as for the special case of \eqref{eq:UiS1} in Theorem~\ref{thm:evenjfixed} with $\beta=0$, and hence omitted. We only remark that $\T_{\wb}(E_i) =E_i$ and thus a Serre-Lusztig relation in $\U^+$ between $E_i$ and $E_j$ gives rise to a same Serre-Lusztig relation in $\U^+$ between $E_i$ and $\T_{\wb} (E_j)$; compare \eqref{eq:gHigherE} and its proof.
\end{proof}

The $\imath$divided powers $B_j^{(n)}$ in $\Ui$, for $j\in \Iw$ such that $\tau j = j \neq \wb j$, were defined in \cite[(5.12)]{BW18c} using $b_j^{(n)}$ in \cite[(5.7)]{BW18c}. In $\tUi$, we modify the definitions to be, for $n\ge 0$,
\begin{align}
b_j^{(n)} &= \sum_{a=0}^n  q_j^{-a(n-a)} ( \T_{w_{\bullet}} (E_j)  \tK'_j)^{(a)}F_j^{(n-a)},
  \label{eq:bn}
\\
B_j^{(n)} &= {b}_j^{(n)}  + \frac{q}{q -q^{-1}} \sum_{k \ge 1}  q_j^{\frac{k(k+1)}{2}} (Z_j \tK_j \tK'_j)^{(k)} b_j^{(n-2k)}.
  \label{eq:Bb}
\end{align}
The following recursive formula holds, for $n\ge 2$ (cf. \cite[(5.8)]{BW18c}):
\begin{align} \label{eq:bbb}
[n]_j b_j^{(n)} = b_j^{(n-1)} B_j - q_j^{2-n} (Z_j \tK_j \tK'_j) b_j^{(n-2)}.
\end{align}

\begin{lem}
  \label{lem:span4}
Let $j\in \Iw$ be such that $\tau j=j \neq \wb j$. Then
  \begin{enumerate}
\item
$B_j^{(n)}$ lies in the $\K(q)$-span of
$ \{ B_j^{n-2\beta} (Z_j \tK_j \tK'_j)^\beta   \mid  0\le \beta \le \lfloor n/2 \rfloor \};$
\item
$B_j^n$ lies in the $\K(q)$-span of
$\{ B_j^{(n-2\beta)} (Z_j \tK_j \tK'_j)^\beta   \mid  0\le \beta \le \lfloor n/2 \rfloor \}.$
\end{enumerate}
\end{lem}

\begin{proof}
Part (1) follows from the definition of $B_{j}^{(n)}$ \eqref{eq:bn}--\eqref{eq:bbb}.
Then, the sets
\[
\{ B_j^{n-2\beta} (Z_j \tK_j \tK'_j)^\beta   \mid  0\le \beta \le \lfloor n/2 \rfloor \}
\text{ and }
\{ B_j^{(n-2\beta)} (Z_j \tK_j \tK'_j)^\beta   \mid  0\le \beta \le \lfloor n/2 \rfloor \}
\]
 span the same vector space $V_n$, and moreover, they are bases of $V_n$. Now Part (2) follows.
\end{proof}

\subsection{Equivalence of \eqref{eq:SerreBn12} and \eqref{eq:SerreBn10}}

We are back to the general setting, and there is no condition on $j$ below.

\begin{prop}
\label{prop:eq:SL-P-iDP}
Suppose $\tau i =\wb i = i \in \Iw$.
Then the identities \eqref{eq:SerreBn12} and \eqref{eq:SerreBn10} are equivalent.
\end{prop}

\begin{proof}
Recall $j\in \Iw$.

(i) Assume $\tau j=j =\wb j$.

We can assume that $a_{ij} \neq 0$, as otherwise 
both \eqref{eq:SerreBn12}--\eqref{eq:SerreBn10} are trivial. Recall $\tk_j$ is central.

\underline{ \eqref{eq:SerreBn12} $\Rightarrow$ \eqref{eq:SerreBn10}}. Let us fix $n\ge 0$. For each $0\le t \le \lfloor n/2 \rfloor$, by \eqref{eq:SerreBn12} with $n$ replaced by $n-2t$, we have
\[
\sum_{r+s=1-na_{ij} + 2ta_{ij}} (-1)^r B^{(r)}_{i,\ov{p}} B_j^{n-2t} B_{i,\ov{p}+\ov{na_{ij}}}^{(s)}  =0.
\]
This together with Proposition~\ref{prop:SLinduct} implies that
$
\sum_{r+s=1-na_{ij}} (-1)^r B^{(r)}_{i,\ov{p}} B_j^{n-2t} B_{i,\ov{p}+\ov{na_{ij}}}^{(s)}  =0.
$
Then \eqref{eq:SerreBn10}  follows, since $B_{j,\ov{t}}^{(n)}$ is a linear combination of $B_j^{n-2t} \tk_j^t$, for $0\le t \le \lfloor n/2 \rfloor$ by Lemma~\ref{lem:span0}.

\underline{ \eqref{eq:SerreBn10} $\Rightarrow$ \eqref{eq:SerreBn12}}. We prove \eqref{eq:SerreBn12} by induction on $n$. The identity~\eqref{eq:SerreBn12} holds for $n=0,1$ as it is the same as \eqref{eq:SerreBn10}. By inductive assumption, we have
\[
\sum_{r+s=1-na_{ij} + 2ta_{ij}} (-1)^r B^{(r)}_{i,\ov{p}} B_j^{n-2t} B_{i,\ov{p}+\ov{na_{ij}}}^{(s)}  =0, \quad
\text{for } 1 \le t \le \lfloor n/2 \rfloor.
\]
Then it follows by Proposition~\ref{prop:SLinduct} that
$
\sum_{r+s=1-na_{ij}} (-1)^r B^{(r)}_{i,\ov{p}} B_j^{n-2t} B_{i,\ov{p}+\ov{na_{ij}}}^{(s)}  =0,
$ 
for $1 \le t \le \lfloor n/2 \rfloor$. Then \eqref{eq:SerreBn12}  follows from these identities and \eqref{eq:SerreBn10} since $B_j^n$ is a linear combination of $B_{j,\ov{t}}^{(n)}$ and $B_j^{n-2t} \tk_j^t$, for $1\le t \le \lfloor n/2 \rfloor$, by Lemma~\ref{lem:span0}.

(ii) Assume $\tau j \neq j$. In this case, the equivalence is trivial since $B_{j,\ov{t}}^{(n)} = B_j^n/[n]_i!$.

(iii) Assume $\tau j=j \neq \wb j$. In this case, the proof is the same as for (i) where Lemma~\ref{lem:span4} is used in place of Lemma~\ref{lem:span0}.
\end{proof}

\subsection{A summary}

Theorem~{\bf A} consists of two identities \eqref{eq:SerreBn12}--\eqref{eq:SerreBn10}. The identity  \eqref{eq:SerreBn12} has been established case-by-case: Theorem~\ref{thm:minLS} (for $\tau j = j =\wb j$), Propositions~\ref{prop:jnotfixed} (for $\tau j \neq j$), and Proposition~\ref{prop:jfixed2} (for $\tau j = j \neq \wb j$).

The identity \eqref{eq:SerreBn10} follows from \eqref{eq:SerreBn12} and the equivalence established in Proposition~\ref{prop:eq:SL-P-iDP}.

Theorem~{\bf B} immediately follows from Theorem~{\bf A} and Proposition~\ref{prop:SLinduct}.

\begin{rem}
Theorem~{\bf A} and Theorem~{\bf B} remain valid over $\Ui =\Ui_\bvs$, once we replace $\tk_i$ by the scalar $\vs_i$ in all relevant places and use the version of $\imath$divided powers in \eqref{eq:iDPoddUi}--\eqref{eq:iDPevUi}.
\end{rem}

\section{General Serre-Lusztig relations for $\tUi$}
  \label{sec:SerreL1}

\subsection{Definition of $\tf_{i,j;n,m,\ov{p},\ov{t},e}$ and $\tf_{i,j;n,m,\ov{p},\ov{t},e}'$}

Let $i\neq j\in \Iw$ be such that $\btau i=i=\wb i$.
Recall the $\imath$divided powers $B_{j,\ov{t}}^{(n)}$ below means $B_{j}^{(n)}$ in cases $j$ satisfies (ii) or (iii) in \S\ref{subsec:main}.

For $m\in\Z$, $n\in\Z_{\geq0}$, $e=\pm1$ and $\ov{p},\ov{t} \in \Z_2$, we define elements $\tf_{i,j;n,m,\ov{p},\ov{t},e}$ and $\tf_{i,j;n,m,\ov{p},\ov{t},e}'$ in $\tUi$ below, depending on the parity of $m-na_{ij}$.

If $m-na_{ij}$ is odd, we let
\begin{align}
\label{eq:m-aodd}
\tf_{i,j;n,m,\ov{p},\ov{t},e} &=\sum_{u\geq0}(q_i\tk_i)^{u} \Big\{
\sum_{\stackrel{ r+s+2u=m}{\ov{r}=\ov{p}+\ov{1}}}
(-1)^r q_i^{-e((m+na_{ij})(r+u)-r)}\qbinom{\frac{m+na_{ij}-1}{2}}{u}_{q_i^2}B^{(r)}_{i,\ov{p}}B_{j,\ov{t}}^{(n)}B^{(s)}_{i,\ov{p}+\ov{na_{ij}}}\\ \notag
&\quad
+\sum_{\stackrel{ r+s+2u=m}{\ov{r}=\ov{p}}}
(-1)^{r}q_i^{-e((m+na_{ij}-2)(r+u)+r)}\qbinom{\frac{m+na_{ij}-1}{2}}{u}_{q_i^2}B^{(r)}_{i,\ov{p}}B_{j,\ov{t}}^{(n)}B^{(s)}_{i,\ov{p}+\ov{na_{ij}}}\Big\} ;
\end{align}
if $m-na_{ij}$ is even, then we let
\begin{align}
\label{eq:m-aeven}
\tf_{i,j;n,m,\ov{p},\ov{t},e} &=\sum_{u\geq0}(q_i\tk_i)^{u}
\Big\{
\sum_{\stackrel{ r+s+2u=m}{ \ov{r}=\ov{p} +\ov{1}}}
(-1)^{r} q_i^{-e(m+na_{ij}-1)(r+u)}
\qbinom{\frac{m+na_{ij}}{2}}{u}_{q_i^2}B^{(r)}_{i,\ov{p}}B_{j,\ov{t}}^{(n)}B^{(s)}_{i,\ov{p}+\ov{na_{ij}}}\\
 &\quad +\sum_{\stackrel{ r+s+2u=m} {\ov{r}=\ov{p} }}
(-1)^{r} q_i^{-e(m+na_{ij}-1)(r+u)}
\qbinom{\frac{m+na_{ij}-2}{2}}{u}_{q_i^2}B^{(r)}_{i,\ov{p}}B_{j,\ov{t}}^{(n)}B^{(s)}_{i,\ov{p}+\ov{na_{ij}}}\Big\}. \notag
\end{align}

If $m-na_{ij}$ is odd, we let
\begin{align}
\label{eq:evev1}
\tf_{i,j;n,m,\ov{p},\ov{t},e}'  &= \sum_{u\geq0}(q_i\tk_i)^{u} \Big\{
\sum_{\stackrel{r+s+2u=m}{\ov{r}=\ov{p}+\ov{1}}}
(-1)^{r}q_i^{-e((m+na_{ij})(r+u)-r)}\qbinom{\frac{m+na_{ij}-1}{2}}{u}_{q_i^2}B^{(s)}_{i,\ov{p}}B_{j,\ov{t}}^{(n)}B^{(r)}_{i,\ov{p}+\ov{na_{ij}}}\\ \notag
&
+\sum_{\stackrel{ r+s+2u=m}{\ov{r}=\ov{p}}}
(-1)^{r}q_i^{-e((m+na_{ij}-2)(r+u)+r)}\qbinom{\frac{m+na_{ij}-1}{2}}{u}_{q_i^2}B^{(s)}_{i,\ov{p}}B_{j,\ov{t}}^{(n)}B^{(r)}_{i,\ov{p}+\ov{na_{ij}}}\Big\};
\end{align}
if $m-na_{ij}$ is even, then we let
\begin{align}
\label{eq:evev0}
\tf_{i,j;n,m,\ov{p},\ov{t},e}' &= \sum_{u\geq0}(q_i\tk_i)^{u}\Big\{
\sum_{\stackrel{ r+s+2u=m}{\ov{r}=\ov{p} +\ov{1}}}
(-1)^{r} q_i^{-e(m+na_{ij}-1)(r+u)}
\qbinom{\frac{m+na_{ij}}{2}}{u}_{q_i^2}B^{(s)}_{i,\ov{p}}B_{j,\ov{t}}^{(n)}B^{(r)}_{i,\ov{p}+\ov{na_{ij}}}\\
 &+\sum_{\stackrel{ r+s+2u=m}{ \ov{r}=\ov{p}}}
(-1)^{r} q_i^{-e(m+na_{ij}-1)(r+u)}
\qbinom{\frac{m+na_{ij}-2}{2}}{u}_{q_i^2}B^{(s)}_{i,\ov{p}}B_{j,\ov{t}}^{(n)}B^{(r)}_{i,\ov{p}+\ov{na_{ij}}}\Big\}. \notag
\end{align}

Recall the anti-involution $\sigma_\imath$ of $\tUi$ in Lemma~\ref{lem:isigma}. The next lemma follows by a direct computation.

\begin{lem}
Let $i \neq j\in \Iw$ be such that $\btau i=i=\wb i$. Then, for any $\ov{p}, \ov{t} \in\Z_2$, $n\ge 0$, $m\in \Z$, and $e=\pm1$, we have
\begin{align}
\label{eq:hos split}
\tf_{i,j;n,m,\ov{p},\ov{t},e}' =\sigma_\imath (\tf_{i, \tau j;n,m,\ov{p},\ov{t},e}).
\end{align}
\end{lem}

\subsection{Recursions and Serre-Lusztig relations}

\begin{thm}
  \label{thm:recursion}
Let $j \neq i \in \Iw$ be such that  $\btau i=i=\wb i$. Then for any $\ov{p}, \ov{t} \in\Z_2$, $n\ge 0$, $m\in \Z$, and $e=\pm1$, we have
\begin{align}
\label{eqn:recursion formula1}
q_i^{-e(2m+na_{ij})} & B_i \tf_{i,j;n,m,\ov{p},\ov{t},e} -\tf_{i,j;n,m,\ov{p},\ov{t},e}B_i\\ \notag
=&-[m+1]_{i}\, \tf_{i,j;n,m+1,\ov{p},\ov{t},e}
+[m+na_{ij}-1]_{i}\, q_i^{1-e(2m+na_{ij}-1)} \tk_i \tf_{i,j;n,m-1,\ov{p},\ov{t},e}.\\ \notag
\\
\label{eqn:recursion formula2}
q_i^{-e(2m+na_{ij})} & \tf'_{i,j;n,m,\ov{p},\ov{t},e}B_i-B_i\tf'_{i,j;n,m,\ov{p},\ov{t},e}\\ \notag
=&-[m+1]_{i}\, \tf'_{i,j;n,m+1,\ov{p},\ov{t},e}
+[m+na_{ij}-1]_{i}\, q_i^{1-e(2m+na_{ij}-1)} \tk_i \tf'_{i,j;n,m-1,\ov{p},\ov{t},e}.
\end{align}
\end{thm}

By applying the anti-involution $\sigma_\imath$, we see that the identities \eqref{eqn:recursion formula1} and  \eqref{eqn:recursion formula2} (with $j$ replaced by $\tau j$) are equivalent.
Hence it suffices to prove  \eqref{eqn:recursion formula1}. The proof of \eqref{eqn:recursion formula1} is divided into the two cases depending on the parity of $m-na_{ij}$, which will occupy \S\ref{subsec:even} and \S\ref{subsec:odd} below, respectively.

\begin{thm}
\label{thm:f=f'=0}
Let $j \neq i\in \Iw$ such that  $\btau i=i=\wb i$, $\ov{p}, \ov{t} \in\Z_2$, $n\ge 0$,  and $e=\pm1$.
Then, for $m<0$ and $m>-na_{ij}$, we have
\begin{align}
\label{eqn:f=f'=0nottrival}
\tf_{i,j;n,m,\ov{p},\ov{t},e}=0,\qquad \tf_{i,j;n,m,\ov{p},\ov{t},e}'=0.
\end{align}
\end{thm}

\begin{proof}
The case for $m<0$ is obvious.

Note that $\tf_{i,j;n,1-na_{ij},\ov{p}, \ov{t}, e}= \sum\limits_{r+s=1-na_{ij}} (-1)^r  B_{i,\overline{a_{ij}}+\overline{p_i}}^{(r)}B_{j,\ov{t}}^{(n)} B_{i,\overline{p}_i}^{(s)}$.
It follows from \eqref{eq:SerreBn1} that
$$\tf_{i,j;n,m,\ov{p},\ov{t},e}=\tf_{i,j;n,m,\ov{p},\ov{t},e}'=0$$
when $m=1-na_{ij}$.
From Theorem~ \ref{thm:recursion} we deduce by induction on $m \geq 1-na_{ij}$, then \eqref{eqn:f=f'=0nottrival} follows.
\end{proof}

\begin{rem}
(1) Theorems~\ref{thm:recursion} and \ref{thm:f=f'=0} hold (with the same proofs) if we replace $B_{j,\ov{t}}^{(n)}$ by $B_j^n$ throughout the definitions of $\tf_{i,j;n,m,\ov{p},\ov{t},e}$ and $\tf_{i,j;n,m,\ov{p},\ov{t},e}'$ in
\eqref{eq:m-aodd}--\eqref{eq:m-aeven} and \eqref{eq:evev1}--\eqref{eq:evev0}. The new variant of Theorem~ \ref{thm:f=f'=0} uses as an initial input the Serre-Lusztig relations of minimal degrees \eqref{eq:SerreBn12} (instead of \eqref{eq:SerreBn10}) in Theorem~{\bf A}.

(2) Theorems~\ref{thm:recursion} and \ref{thm:f=f'=0} remain valid over $\Ui =\Ui_\bvs$, once we replace $\tk_i$ by the scalar $\vs_i$ in all relevant places and use the version of $\imath$divided powers in \eqref{eq:iDPoddUi}--\eqref{eq:iDPevUi}.
\end{rem}

\subsection{Braid group symmetries for $\tUi$}

Let $i\neq j\in \Iw$ be such that $\btau i=i=\wb i$. The definitions of $\tf_{i,j;n,m,\ov{p},\ov{t},e}$ in \eqref{eq:m-aeven} and $\tf'_{i,j;n,m,\ov{p},\ov{t},e}$ in \eqref{eq:evev0} dramatically simplify when $m =-na_{ij}$ as follows:
\begin{align}
 \label{eq:yy}
\tf_{i,j;n,-na_{ij},\ov{p},\ov{t},e}
=&\sum_{r+s=-na_{ij}}(-1)^rq_i^{er}B_{i,\ov{p}}^{(r)}B_{j,\ov{t}}^{(n)}B_{i,\ov{na_{ij}}+\ov{p}}^{(s)}\\
&+\sum_{u\geq1}\sum_{\stackrel{r+s+2u=-na_{ij},}{\ov{r}=\ov{p}}}(-1)^{r+u}q_i^{er+eu} (q_i\tk_i)^u B_{i,\ov{p}}^{(r)}B_{j,\ov{t}}^{(n)}B_{i,\ov{na_{ij}}+\ov{p}}^{(s)},
\notag
\\
\tf'_{i,j;n,-na_{ij},\ov{p},\ov{t},e}
=&\sum_{r+s=-na_{ij}}(-1)^rq_i^{er} B_{i,\ov{na_{ij}}+\ov{p}}^{(s)}B_{j,\ov{t}}^{(n)} B_{i,\ov{p}}^{(r)}
 \label{eq:yy2} \\
&+\sum_{u\geq1}\sum_{\stackrel{r+s+2u=-na_{ij},}{\ov{r}=\ov{p}}}(-1)^{r+u}q_i^{er+eu} (q_i\tk_i)^u B_{i,\ov{na_{ij}}+\ov{p}}^{(s)}B_{j,\ov{t}}^{(n)} B_{i,\ov{p}}^{(r)}.
\notag
\end{align}


\begin{conj}
\label{conj:braid}
For each $i\in \Iw$ such that $\tau i =i =\wb i$, there exist automorphisms $\T_{i,e}'$ $\T_{i,-e}''$ of $\tUi$ , which are inverses to each other, such that

\begin{align*}
\T_{i,e}'(\tk_j) &=(-q_i^{1+e} \tk_i)^{-a_{ij}}\tk_j =\T_{i,-e}''(\tk_j),\quad \forall j\in \I;
\\
\T_{i,e}'(B_j) &=  \begin{cases} \tf_{i,j;1,-a_{ij},\ov{p},\ov{t},e} & \text{ if }i\neq j \in \Iw,
\\
(-q_i^{1+e}\tk_i)^{-1} B_i,  &\text{ if } i=j;  \end{cases}
\\
\T_{i,-e}''(B_j) &= \begin{cases} \tf'_{i,j;1,-a_{ij},\ov{p},\ov{t},e} & \text{ if } i\neq j \in \Iw,
\\
(-q_i^{1+e}\tk_i)^{-1} B_i,  &\text{ if } i=j.  \end{cases}
\end{align*}

\end{conj}

\begin{rem}
Let $i\in \Iw$ such that $\tau i =i =\wb i$. 
By choosing the distinguished parameters such that $\vs_i=-q_i^{-2}$ (cf. \cite[\S7]{LW19b}), $\T_{i,1}'$ and $\T_{i,-1}''$ induce the braid group symmetries on $\Ui$, which are denoted by the same notations. The $\T'_{i,1}$ and $\T_{i,-1}''$ on $\Ui$
coincide with the braid group operators (denoted by $\tau_i$ and $\tau_i^{-}$) obtained earlier for split $\Ui$ of finite type in \cite{KP11} and split affine type $A_1$ in \cite{BaK20} via a computer computation (also cf. \cite{Ter18} for a computer free verification).
\end{rem}

In a separate publication, we shall develop a Hall algebra approach (cf. \cite{LW19b}) to prove Conjecture~\ref{conj:braid} for quasi-split $\imath$quantum groups. We shall also establish various favorable properties of these symmetries, in a way strikingly parallel to those of Lusztig for $\U$ \cite{Lu93}.

\subsection{Proof of Theorem~\ref{thm:recursion} for $m-na_{ij}$ even}
  \label{subsec:even}

Using \eqref{lem:dividied power} and the definitions of $\tf_{i,j;n,m,\ov{p},\ov{t},e}$ in
\eqref{eq:m-aodd}--\eqref{eq:m-aeven}, we have
\begin{align}
\label{eq:qcomm}
&q_i^{-e(2m+na_{ij})}B_i\tf_{i,j;n,m,\ov{p},\ov{t},e}-\tf_{i,j;n,m,\ov{p},\ov{t},e}B_i\\ \notag
=&\sum_{u\geq0}(q_i\tk_i)^u\\   \notag
&  \cdot\Big\{
\sum_{\stackrel{ r+s+2u=m} {\ov{r}=\ov{p} +\ov{1}}}
(-1)^r q_i^{-e(m+na_{ij}-1)(r+u)-e(2m+na_{ij})} [r+1]_{i}
\qbinom{\frac{m+na_{ij}}{2}}{u}_{q_i^2}
B^{(r+1)}_{i,\ov{p}}B_{j,\ov{t}}^{(n)}B^{(s)}_{i,\ov{p}+\ov{na_{ij}}} \\ \notag
&
+\sum_{\stackrel{ r+s+2u=m}{ \ov{r}=\ov{p}}}
(-1)^rq_i^{-e(m+na_{ij}-1)(r+u)-e(2m+na_{ij})}[r+1]_{i}
\qbinom{\frac{m+na_{ij}-2}{2}}{u}_{q_i^2}
B^{(r+1)}_{i,\ov{p}}B_{j,\ov{t}}^{(n)}B^{(s)}_{i,\ov{p}+\ov{na_{ij}}}\\ \notag
& -\sum_{\stackrel{ r+s+2u=m} {\ov{r}=\ov{p} +\ov{1}}}
(-1)^rq_i^{-e(m+na_{ij}-1)(r+u)}[s+1]_{i}
\qbinom{\frac{m+na_{ij}}{2}}{u}_{q_i^2}
B^{(r)}_{i,\ov{p}}B_{j,\ov{t}}^{(n)}B^{(s+1)}_{i,\ov{p}+\ov{na_{ij}}} \\ \notag
&
-\sum_{\stackrel{ r+s+2u=m} {\ov{r}=\ov{p} }}
(-1)^rq_i^{-e(m+na_{ij}-1)(r+u)}[s+1]_{i}
\qbinom{\frac{m+na_{ij}-2}{2}}{u}_{q_i^2}
B^{(r)}_{i,\ov{p}}B_{j,\ov{t}}^{(n)}B^{(s+1)}_{i,\ov{p}+\ov{na_{ij}}} \\ \notag
&\quad +q_i\tk_i\Big(
\sum_{\stackrel{ r+s+2u=m}{\ov{r}=\ov{p}}}
(-1)^rq_i^{-e(m+na_{ij}-1)(r+u)-e(2m+na_{ij})}[r]_{i}
\qbinom{\frac{m+na_{ij}-2}{2}}{u}_{q_i^2}
B^{(r-1)}_{i,\ov{p}}B_{j,\ov{t}}^{(n)}B^{(s)}_{i,\ov{p}+\ov{na_{ij}}}\\ \notag
& -\sum_{\stackrel{ r+s+2u=m}{ \ov{r}=\ov{p}}}
(-1)^rq_i^{-e(m+na_{ij}-1)(r+u)}[s]_{i}
\qbinom{\frac{m+na_{ij}-2}{2}}{u}_{q_i^2}
B^{(r)}_{i,\ov{p}}B_{j,\ov{t}}^{(n)}B^{(s-1)}_{i,\ov{p}+\ov{na_{ij}}}\Big)\Big\}
\\
=& - \sum_{u\geq0}(q_i\tk_i)^u  ( X_1 +X_2 + X_3 +X_4 +X_5 +X_6 ),
\notag
\end{align}
where
\begin{align*}
&
X_1 =
\sum_{\stackrel{ r+s+2u=m+1}{ \ov{r}=\ov{p}}}
(-1)^{r}q_i^{-e(m+na_{ij}-1)(r+u)-e(m+1)}[r]_{i} \qbinom{\frac{m+na_{ij}}{2}}{u}_{q_i^2}B^{(r)}_{i,\ov{p}}B_{j,\ov{t}}^{(n)}B^{(s)}_{i,\ov{p}+\ov{na_{ij}}}, 
\\
&
X_2 =
\sum_{\stackrel{ r+s+2u=m+1}{ \ov{r}=\ov{p}+\ov{1} }}
(-1)^rq_i^{-e(m+na_{ij}-1)(r+u)-e(m+1)}[r]_{i} \qbinom{\frac{m+na_{ij}-2}{2}}{u}_{q_i^2}B^{(r)}_{i,\ov{p}}B_{j,\ov{t}}^{(n)}B^{(s)}_{i,\ov{p}+\ov{na_{ij}}}, 
\\
&
X_3 =
\sum_{\stackrel{ r+s+2u=m+1}{\ov{r}=\ov{p} +\ov{1}}}
(-1)^rq_i^{-e(m+na_{ij}-1)(r+u)}[s]_{i} \qbinom{\frac{m+na_{ij}}{2}}{u}_{q_i^2}B^{(r)}_{i,\ov{p}}B_{j,\ov{t}}^{(n)}B^{(s)}_{i,\ov{p}+\ov{na_{ij}}}, 
\\
&
X_4 =
\sum_{\stackrel{ r+s+2u=m+1}{ \ov{r}=\ov{p}}}
(-1)^rq_i^{-e(m+na_{ij}-1)(r+u)}[s]_{i} \qbinom{\frac{m+na_{ij}-2}{2}}{u}_{q_i^2}B^{(r)}_{i,\ov{p}}B_{j,\ov{t}}^{(n)}B^{(s)}_{i,\ov{p}+\ov{na_{ij}}},
\end{align*}
and
\begin{align*}
&
X_5 =
\sum_{\stackrel{ r+s+2u=m+1}{ \ov{r}=\ov{p}+\ov{1} }}
(-1)^rq_i^{-e(m+na_{ij}-1)(r+u)-e(2m+na_{ij})}
[r+1]_{i} \qbinom{\frac{m+na_{ij}-2}{2}}{u-1}_{q_i^2}B^{(r)}_{i,\ov{p}}B_{j,\ov{t}}^{(n)}B^{(s)}_{i,\ov{p}+\ov{na_{ij}}},
\notag \\
&
X_6 =
\sum_{\stackrel{ r+s+2u=m+1}{ \ov{r}=\ov{p}}}
(-1)^rq_i^{-e(m+na_{ij}-1)(r+u-1)}[s+1]_{i} \qbinom{\frac{m+na_{ij}-2}{2}}{u-1}_{q_i^2}B^{(r)}_{i,\ov{p}}B_{j,\ov{t}}^{(n)}B^{(s)}_{i,\ov{p}+\ov{na_{ij}}}.
\end{align*}

We compute the partial sum
\begin{align*}
&\sum_{u\geq0}(q_i\tk_i)^u (X_1 +X_4 +X_6)
= \sum_{u\geq0}(q_i\tk_i)^u
\sum_{\stackrel{ r+s+2u=m+1}{ \ov{r}=\ov{p}}}(-1)^rq_i^{-e(m+na_{ij}-1)(r+u-1)}\\ \notag
&\cdot \Big\{ q_i^{-e(2m+na_{ij})}[r]_{i} \qbinom{\frac{m+na_{ij}}{2}}{u}_{q_i^2}
+
q_i^{-e(m+na_{ij}-1)}[s]_{i} \qbinom{\frac{m+na_{ij}-2}{2}}{u}_{q_i^2}
+[s+1]_{i} \qbinom{\frac{m+na_{ij}-2}{2}}{u-1}_{q_i^2}\Big\}\\
&\cdot B^{(r)}_{i,\ov{p}}B_{j,\ov{t}}^{(n)}B^{(s)}_{i,\ov{p}+\ov{na_{ij}}}. \notag
\end{align*}
By Lemma \ref{lem:eqn1} (with $a=a_{ij}$), we rewrite the above as
\begin{align}\label{eqn:h146}
&\sum_{u\geq 0} (q_i\tk_i)^u (X_1 +X_4 +X_6)
\\ \notag
=& [m+1]_{i} \sum_{u\geq0}(q_i\tk_i)^u
\sum_{\stackrel{ r+s+2u=m+1}{ \ov{r}=\ov{p}}}(-1)^r q_i^{-e((m+na_{ij}-1)(r+u)+r)}\qbinom{\frac{m+na_{ij}}{2}}{u}_{q_i^2}B^{(r)}_{i,\ov{p}}B_{j,\ov{t}}^{(n)}B^{(s)}_{i,\ov{p}+\ov{na_{ij}}}\\ \notag
& - [m+na_{ij}-1]_{i} q_i^{-e(2m+na_{ij}-1)}\sum_{u\geq 0}(q_i\tk_i)^{u+1}\sum_{\stackrel{ r+s+2u=m-1}{ \ov{r}=\ov{p} }}(-1)^r\\ \notag
& \quad \cdot q_i^{-e((m+na_{ij}-3)(u+r)+r)}
\qbinom{\frac{m+na_{ij}-2}{2}}{u}_{q_i^2}
B^{(r)}_{i,\ov{p}}B_{j,\ov{t}}^{(n)}B^{(s)}_{i,\ov{p}+\ov{na_{ij}}}.
\end{align}

Similarly, we compute another partial sum
\begin{align*}
&\sum_{u\geq 0}(q_i\tk_i)^u (X_2 +X_3 +X_5)
 \\ \notag
=&
\sum_{u\geq0}(q_i\tk_i)^u \sum_{\stackrel{ r+s+2u=m+1}{ \ov{r}=\ov{p}+\ov{1} }}(-1)^r
q_i^{-e((m+na_{ij}-1)(r+u)+2m+na_{ij})}\\ \notag
&\cdot \left\{q_i^{e(2m+na_{ij})}[s]_{i} \qbinom{\frac{m+na_{ij}}{2}}{u}_{q_i^2}+q_i^{e(m+na_{ij}-1)}[r]_{i} \qbinom{\frac{m+na_{ij}-2}{2}}{u}_{q_i^2}+
[r+1]_{i} \qbinom{\frac{m+na_{ij}-2}{2}}{u-1}_{q_i^2}
\right\}
\\
&\cdot B^{(r)}_{i,\ov{p}}B_{j,\ov{t}}^{(n)}B^{(s)}_{i,\ov{p}+\ov{na_{ij}}}.
\end{align*}
By Lemma \ref{lem:eqn1} (with $a=a_{ij}$), we continue to rewrite the above as
\begin{align}\label{eqn:h235}
&\sum_{u\geq0}(q_i\tk_i)^u  (X_2 +X_3 +X_5)
\\ \notag
=&
[m+1]_{i} \sum_{u\geq 0}(q_i\tk_i)^u \sum_{\stackrel{ r+s+2u=m+1}{ \ov{r}=\ov{p}+\ov{1} }}
(-1)^rq_i^{-e((m+na_{ij}+1)(r+u)-r)}\qbinom{\frac{m+na_{ij}}{2}}{u}_{q_i^2}B^{(r)}_{i,\ov{p}}B_{j,\ov{t}}^{(n)}B^{(s)}_{i,\ov{p}+\ov{na_{ij}}}\\\notag
& -[m+na_{ij}-1]_{i} q_i^{-e(2m+na_{ij}-1)}\sum_{u\geq0} (q_i\tk_i)^{u+1} \sum_{\stackrel{r+s+2u=m-1}{ r=\ov{p}+\ov{1} }}(-1)^r  \\ \notag
 &\quad\quad \cdot q_i^{-e((m+na_{ij}-1)(r+u)-r)}\qbinom{\frac{m+na_{ij}-2}{2}}{u}_{q_i^2}
B^{(r)}_{i,\ov{p}}B_{j,\ov{t}}^{(n)}B^{(s)}_{i,\ov{p}+\ov{na_{ij}}}.
\end{align}
Plugging \eqref{eqn:h146} and \eqref{eqn:h235} into \eqref{eq:qcomm}, by using \eqref{eq:m-aodd}--\eqref{eq:m-aeven}, we obtain
\begin{align*}
q_i^{-e(2m+na_{ij})} & B_i\tf_{i,j;n,m,\ov{p},\ov{t},e}-\tf_{i,j;n,m,\ov{p},\ov{t},e}B_i\\
=&-[m+1]_{i} \tf_{i,j;n,m+1,\ov{p},\ov{t},e}
+[m+na_{ij}-1]_{i} q_i^{1-e(2m+na_{ij}-1)} \tk_i \tf_{i,j;n,m-1,\ov{p},\ov{t},e}.
\end{align*}

\subsection{Proof of Theorem~\ref{thm:recursion} for $m-na_{ij}$ odd}
  \label{subsec:odd}

Using \eqref{lem:dividied power} and the definitions of $\tf_{i,j;n,m,\ov{p},\ov{t},e}$ in
\eqref{eq:m-aodd}--\eqref{eq:m-aeven}, we have
\begin{align}
 \label{eq:qcomm2}
&q_i^{-e(2m+na_{ij})}B_i \tf_{i,j;n,m,\ov{p},\ov{t},e}-\tf_{i,j;n,m,\ov{p},\ov{t},e}B_i\\ \notag
=&\sum_{u\geq0}(q_i\tk_i)^u\\ \notag
&\cdot \Big\{
\sum_{\stackrel{ r+s+2u=m}{ \ov{r}=\ov{p}+\ov{1} }}(-1)^r
q_i^{-e((m+na_{ij})(r+u)-r+2m+na_{ij})}[r+1]_{i} \qbinom{\frac{m+na_{ij}-1}{2}}{u}_{q_i^2}B^{(r+1)}_{i,\ov{p}}B_{j,\ov{t}}^{(n)}B^{(s)}_{i,\ov{p}+\ov{na_{ij}}} \\
& \notag
\quad
+\sum_{\stackrel{ r+s+2u=m}{ \ov{r}=\ov{p}}}(-1)^r
q_i^{-e((m+na_{ij}-2)(r+u)+r+2m+na_{ij})}[r+1]_{i} \qbinom{\frac{m+na_{ij}-1}{2}}{u}_{q_i^2}B^{(r+1)}_{i,\ov{p}}B_{j,\ov{t}}^{(n)}B^{(s)}_{i,\ov{p}+\ov{na_{ij}}} \\
& \notag
-\sum_{\stackrel{ r+s+2u=m}{ \ov{r}=\ov{p}+\ov{1} }}(-1)^r
q_i^{-e((m+na_{ij})(r+u)-r)}[s+1]_{i} \qbinom{\frac{m+na_{ij}-1}{2}}{u}_{q_i^2}B^{(r)}_{i,\ov{p}}B_{j,\ov{t}}^{(n)}B^{(s+1)}_{i,\ov{p}+\ov{na_{ij}}}\\
& \notag
-\sum_{\stackrel{ r+s+2u=m}{ \ov{r}=\ov{p}}}(-1)^r
q_i^{-e((m+na_{ij}-2)(r+u)+r)}[s+1]_{i} \qbinom{\frac{m+na_{ij}-1}{2}}{u}_{q_i^2}B^{(r)}_{i,\ov{p}}B_{j,\ov{t}}^{(n)}B^{(s+1)}_{i,\ov{p}+\ov{na_{ij}}}\\
& \notag
+(q_i\tk_i)\Big(\sum_{\stackrel{ r+s+2u=m}{ \ov{r}=\ov{p}}}(-1)^r
q_i^{-e((m+na_{ij}-2)(r+u)+r+2m+na_{ij})}[r]_{i} \qbinom{\frac{m+na_{ij}-1}{2}}{u}_{q_i^2}B^{(r-1)}_{i,\ov{p}}B_{j,\ov{t}}^{(n)}B^{(s)}_{i,\ov{p}+\ov{na_{ij}}}\\
& \notag
\quad \quad -\sum_{\stackrel{ r+s+2u=m}{ \ov{r}=\ov{p}+\ov{1} }}(-1)^r
q_i^{-e((m+na_{ij})(r+u)-r)}[s]_{i} \qbinom{\frac{m+na_{ij}-1}{2}}{u}_{q_i^2}B^{(r)}_{i,\ov{p}}B_jB^{(s-1)}_{i,\ov{p}+\ov{a_{ij}}}\Big)\Big\}
\\
=& - \sum_{u\geq 0}(q_i\tk_i)^u (Y_1 + Y_2 + Y_3 + Y_4 + Y_5 + Y_6),
\notag
\end{align}
where
\begin{align*}
&
Y_1 =
\sum_{\stackrel{ r+s+2u=m+1}{ \ov{r}=\ov{p}}}(-1)^r
q_i^{-e((m+na_{ij})(r+u)-r+m+1)}[r]_{i} \qbinom{\frac{m+na_{ij}-1}{2}}{u}_{q_i^2}B^{(r)}_{i,\ov{p}}B_{j,\ov{t}}^{(n)}B^{(s)}_{i,\ov{p}+\ov{na_{ij}}}, \\
&
Y_2 =
\sum_{\stackrel{ r+s+2u=m+1}{ \ov{r}=\ov{p}+\ov{1} }}(-1)^r
q_i^{-e((m+na_{ij}-2)(r+u)+r+1+m)}[r]_{i} \qbinom{\frac{m+na_{ij}-1}{2}}{u}_{q_i^2}B^{(r)}_{i,\ov{p}}B_{j,\ov{t}}^{(n)}B^{(s)}_{i,\ov{p}+\ov{na_{ij}}},\\
&
Y_3 =
\sum_{\stackrel{ r+s+2u=m+1}{ \ov{r}=\ov{p}+\ov{1} }}(-1)^r
q_i^{-e((m+na_{ij})(r+u)-r)}[s]_{i} \qbinom{\frac{m+na_{ij}-1}{2}}{u}_{q_i^2}B^{(r)}_{i,\ov{p}}B_{j,\ov{t}}^{(n)}B^{(s)}_{i,\ov{p}+\ov{na_{ij}}}, \\
&
Y_4 =
\sum_{\stackrel{ r+s+2u=m+1}{ \ov{r}=\ov{p}}}(-1)^r
q_i^{-e((m+na_{ij}-2)(r+u)+r)}[s]_{i} \qbinom{\frac{m+na_{ij}-1}{2}}{u}_{q_i^2}B^{(r)}_{i,\ov{p}}B_{j,\ov{t}}^{(n)}B^{(s)}_{i,\ov{p}+\ov{na_{ij}}}
\Big\},
\end{align*}
and
\begin{align*}
&
Y_5 =
\sum_{\stackrel{ r+s+2u=m+1}{ \ov{r}=\ov{p}+\ov{1} }}(-1)^r
q_i^{-e((m+na_{ij}-2)(r+u)+r+1+2m+na_{ij})}
[r+1]_{i} \qbinom{\frac{m+na_{ij}-1}{2}}{u-1}_{q_i^2}B^{(r)}_{i,\ov{p}}B_{j,\ov{t}}^{(n)}B^{(s)}_{i,\ov{p}+\ov{na_{ij}}},
\notag \\
&
Y_6 =
\sum_{\stackrel{ r+s+2u=m+1}{ \ov{r}=\ov{p}+\ov{1} }}(-1)^r
q_i^{-e((m+na_{ij})(r+u-1)-r)}[s+1]_{i} \qbinom{\frac{m+na_{ij}-1}{2}}{u-1}_{q_i^2}B^{(r)}_{i,\ov{p}}B_{j,\ov{t}}^{(n)}B^{(s)}_{i,\ov{p}+\ov{na_{ij}}}.
\end{align*}

Note that
\begin{align}
\label{eq:qbinom}
q_i^{-e(m+1+2u)}[r]_{i} +[s]_{i} =[m+1]_{i} q_i^{-e(2u+r)}-q_i^{-e(m+1-r)}[2u]_{i},
\end{align}
where $r+s=m+1-2u$, and
\begin{align}
\label{eq:qbinom2}
[2u]_{i} \qbinom{\frac{m+na_{ij}-1}{2}}{u}_{q_i^2}=[m+na_{ij}-1]_{i} \qbinom{\frac{m+na_{ij}-3}{2}}{u-1}_{q_i^2}
\end{align}
for $u\geq1$. We compute a partial sum
\begin{align}\label{eqn:h'14}
&\sum_{u\geq0}(q_i\tk_i)^u (Y_1 +Y_4)
\\  \notag
=&
\sum_{u\geq0}(q_i\tk_i)^u\Big\{
\sum_{\stackrel{ r+s+2u=m+1}{ \ov{r}=\ov{p}}}(-1)^r
q_i^{-e((m+na_{ij}-2)(r+u)+r)}(q_i^{-e(m+1+2u)}[r]_{i} +[s]_{i} )\\\notag
& \quad \quad \quad \quad \cdot\qbinom{\frac{m+na_{ij}-1}{2}}{u}_{q_i^2}B^{(r)}_{i,\ov{p}}B_{j,\ov{t}}^{(n)}B^{(s)}_{i,\ov{p}+\ov{na_{ij}}}\Big\}\\\notag
=&
[m+1]_{i} \sum_{u\geq0}(q_i\tk_i)^u
\sum_{\stackrel{ r+s+2u=m+1}{ \ov{r}=\ov{p}}}(-1)^r
q_i^{-e(m+na_{ij})(r+u)}\qbinom{\frac{m+na_{ij}-1}{2}}{u}_{q_i^2}B^{(r)}_{i,\ov{p}}B_{j,\ov{t}}^{(n)}B^{(s)}_{i,\ov{p}+\ov{na_{ij}}} \\\notag
&
 -q_i^{-e(m+1)}\sum_{u\geq0}(q_i\tk_i)^u\sum_{\stackrel{ r+s+2u=m+1}{ \ov{r}=\ov{p}}}(-1)^r
q_i^{-e(m+na_{ij}-2)(r+u)}[2u]_{i} \qbinom{\frac{m+na_{ij}-1}{2}}{u}_{q_i^2}B^{(r)}_{i,\ov{p}}B_{j,\ov{t}}^{(n)}B^{(s)}_{i,\ov{p}+\ov{na_{ij}}}\\\notag
=&[m+1]_{i} \sum_{u\geq0}(q_i\tk_i)^u
\sum_{\stackrel{ r+s+2u=m+1}{ \ov{r}=\ov{p}}}(-1)^r
q_i^{-e(m+na_{ij})(r+u)}\qbinom{\frac{m+na_{ij}-1}{2}}{u}_{q_i^2}B^{(r)}_{i,\ov{p}}B_{j,\ov{t}}^{(n)}B^{(s)}_{i,\ov{p}+\ov{na_{ij}}}\\\notag
&
 -q_i^{-e(m+1)}\sum_{u\geq1}(q_i\tk_i)^u\sum_{\stackrel{ r+s+2u=m+1}{ \ov{r}=\ov{p}}}(-1)^r \\\notag
& \quad \quad \quad \quad
\cdot q_i^{-e(m+na_{ij}-2)(r+u)}[m+na_{ij}-1]_{i} \qbinom{\frac{m+na_{ij}-3}{2}}{u-1}_{q_i^2}B^{(r)}_{i,\ov{p}}B_{j,\ov{t}}^{(n)}B^{(s)}_{i,\ov{p}+\ov{na_{ij}}}
\end{align}
\begin{align*}
=&[m+1]_{i} \sum_{u\geq0}(q_i\tk_i)^u
\sum_{\stackrel{ r+s+2u=m+1}{ \ov{r}=\ov{p}}}(-1)^r
q_i^{-e(m+na_{ij})(r+u)}\qbinom{\frac{m+na_{ij}-1}{2}}{u}_{q_i^2}B^{(r)}_{i,\ov{p}}B_{j,\ov{t}}^{(n)}B^{(s)}_{i,\ov{p}+\ov{na_{ij}}}\\\notag
&
-[m+na_{ij}-1]_{i} q_i^{1-e(2m+na_{ij}-1)}\tk_i\sum_{u\geq0}(q_i\tk_i)^u\sum_{\stackrel{ r+s+2u=m-1}{ \ov{r}=\ov{p} }}(-1)^r\\ \notag
& \quad \quad \quad \quad
 \cdot q_i^{-e(m+na_{ij}-2)(r+u)}\qbinom{\frac{m+na_{ij}-3}{2}}{u}_{q_i^2}B^{(r)}_{i,\ov{p}}B_{j,\ov{t}}^{(n)}B^{(s)}_{i,\ov{p}+\ov{na_{ij}}}.
\end{align*}
Here the second equality follows by using \eqref{eq:qbinom}, the third equality follows by \eqref{eq:qbinom2}.

Now we compute another partial sum
\begin{align*}
 \sum_{u\geq0} & (q_i\tk_i)^u (Y_2 + Y_3 + Y_5 + Y_6)
\\ \notag
=&
\sum_{u\geq0}(q_i\tk_i)^u\sum_{\stackrel{ r+s+2u=m+1}{ \ov{r}=\ov{p}+\ov{1} }}(-1)^rq_i^{-e((m+na_{ij}-2)(r+u-1)+r)}\\ \notag
& \cdot
\Big \{q_i^{-e(2m+na_{ij}-1)}[r]_{i} \qbinom{\frac{m+na_{ij}-1}{2}}{u}_{q_i^2}
+q_i^{-e(2u+m+na_{ij}-2)}[s]_{i} \qbinom{\frac{m+na_{ij}-1}{2}}{u}_{q_i^2}\\ \notag
&\quad  +q_i^{-e(3m+2na_{ij}-1)}[r+1]_{i} \qbinom{\frac{m+na_{ij}-1}{2}}{u-1}_{q_i^2}
+q_i^{-e(2u-2)}[s+1]_{i} \qbinom{\frac{m+na_{ij}-1}{2}}{u-1}_{q_i^2}\Big\}
\\
& \notag
\cdot B^{(r)}_{i,\ov{p}}B_{j,\ov{t}}^{(n)}B^{(s)}_{i,\ov{p}+\ov{na_{ij}}}.
\end{align*}

By Lemma \ref{lem:eqn2} (with $a=a_{ij}$), we rewrite the above as
\begin{align}\label{eqn:h'od2356}
&\sum_{u\geq0}(q_i\tk_i)^u (Y_2 + Y_3 + Y_5 + Y_6)
\\ \notag
=&
[m+1]_{i} \sum_{u\geq0}(q_i\tk_i)^u\sum_{\stackrel{ r+s+2u=m+1}{ \ov{r}=\ov{p}+\ov{1} }}(-1)^r
q_i^{-e(m+na_{ij})(r+u)}\qbinom{\frac{m+na_{ij}+1}{2}}{u}_{q_i^2}B^{(r)}_{i,\ov{p}}B_{j,\ov{t}}^{(n)}B^{(s)}_{i,\ov{p}+\ov{na_{ij}}}\\ \notag
& -[m+na_{ij}-1]_{i} q_i^{-e(2m+na_{ij}-1)}
\sum_{u\geq0}(q_i\tk_i)^{u}\sum_{\stackrel{r+s+2u=m-1}{ \ov{r}=\ov{p}+\ov{1} }}(-1)^r\\ \notag
& \quad \quad \quad \quad \cdot
q_i^{-e(m+na_{ij}-2)(r+u-1)}\qbinom{\frac{m+na_{ij}-1}{2}}{u-1
}_{q_i^2}B^{(r)}_{i,\ov{p}}B_{j,\ov{t}}^{(n)}B^{(s)}_{i,\ov{p}+\ov{na_{ij}}}\\ \notag
=&
[m+1]_{i} \sum_{u\geq0}(q_i\tk_i)^u\sum_{\stackrel{ r+s+2u=m+1}{ \ov{r}=\ov{p}+\ov{1} }}(-1)^r
q_i^{-e(m+na_{ij})(r+u)}\qbinom{\frac{m+na_{ij}+1}{2}}{u}_{q_i^2}B^{(r)}_{i,\ov{p}}B_{j,\ov{t}}^{(n)}B^{(s)}_{i,\ov{p}+\ov{na_{ij}}}\\ \notag
& -[m+na_{ij}-1]_{i} q_i^{-e(2m+na_{ij}-1)}
\sum_{u\geq0}(q_i\tk_i)^{u+1}\sum_{\stackrel{r+s+2u=m-1}{ \ov{r}=\ov{p}+\ov{1} }}(-1)^r\\ \notag
& \quad \quad \quad \quad \cdot
q_i^{-e(m+na_{ij}-2)(r+u)}\qbinom{\frac{m+na_{ij}-1}{2}}{u}_{q_i^2}B^{(r)}_{i,\ov{p}}B_{j,\ov{t}}^{(n)}B^{(s)}_{i,\ov{p}+\ov{na_{ij}}}.
\end{align}

Plugging \eqref{eqn:h'14}--\eqref{eqn:h'od2356} into \eqref{eq:qcomm2} and then using \eqref{eq:m-aodd}--\eqref{eq:m-aeven}, we obtain
\begin{align*}
q^{-e(2m+na_{ij})} & B_i \tf_{i,j;n,m,\ov{p},\ov{t},e}-\tf_{i,j;n,m,\ov{p},\ov{t},e}B_i\\
=-&[m+1]_{i} \tf_{i,j;n,m+1,\ov{p},\ov{t},e}
+[m+na_{ij}-1]_{i} q_i^{1-e(2m+na_{ij}-1)} \tk_i \tf_{i,j;n,m-1,\ov{p},\ov{t},e}.
\end{align*}
The proof of Theorem~\ref{thm:recursion} is completed.

%

\appendix

\section{More reductions from Serre-Lusztig}
    \label{App:A}

In this appendix, we outline the proofs of the identities \eqref{eq:serre11odd}--\eqref{eq:serre11evenodd}, which are modeled on the proof of \eqref{eq:serre11F} in \S\ref{subsec:proof1}--\ref{subsec:evev}.


\subsection{Proof of the identity \eqref{eq:serre11odd} }

Let $\alpha =-a_{12}$ as before. We shall use (\ref{t2mdot2})--(\ref{t2m-1dot2}) to rewrite the element
\begin{align}
\label{eq:Serre-idem2}
\sum_{r=0}^{\alpha(\mu+\nu+2\beta)+1} (-1)^r  B_{1,\odd}^{(r)}E_2^\mu K_2^{-(\mu+2\beta)}F_2^\nu  B_{1,\odd}^{(\alpha(\mu+\nu+2\beta)+1-r)}\onestar_{2\la-1} \in \Udot
\end{align}
for any $\la\in\Z$ in terms of monomials in $E_1, F_1, E_2, F_2, \tK_2^{-1}$.

Similar to \eqref{eq:evev}, we obtain the following formula:
\begin{align*}
& \sum_{r=0}^{\alpha(\mu+\nu+2\beta)+1}
 (-1)^r B_{1,\odd}^{(r)}E_2^\mu K_2^{-(\mu+2\beta)}F_2^\nu B_{1,\odd}^{(\alpha(\mu+\nu+2\beta)+1-r)}\onestar_{2\la-1} =A_1 -A_2
 \end{align*}
 where
\begin{align*}
A_1& =\sum_{r=0,2\mid r}^{\alpha(\mu+\nu+2\beta)+1}\sum_{c=0}^{\frac{\alpha}{2}(\mu+\nu+2\beta)-\frac{r}{2}}\sum_{e=0}^{\frac{r}{2}} \sum_{a=0}^{\alpha(\mu+\nu+2\beta)+1-r-2c}\sum_{d=0}^{r-2e}\sum^{\min\{a,r-2e-d\}}_{b=0}
\\
& q_1^{(a+c+d+e)(\alpha(\mu+\nu+2\beta)+3-r-2\la-2a-2c-2d-2e)+2\alpha(\mu+\beta)(d+e)-a-2c}q_1^{\alpha(\mu+2\beta)(r+a-b-2e-d)}\notag\\
& \cdot
\qbinom{2\alpha\mu +\alpha\nu+4\alpha\beta+3-2e-d-3a-2\la-4c-r}{b}_{q_1}  \notag \\
&  \cdot  \qbinom{\frac{\alpha}{2}(\mu+\nu+2\beta)-\frac{r}{2}-c-a-\la+1}{c}_{q_1^2}
\qbinom{\frac{3\alpha\mu}{2} +\frac{\alpha\nu}{2}+2\alpha\beta+1-e-d-\la-2a-\frac{r}{2}-2c}{e}_{q_1^2}
\notag \\
&  \cdot E_1^{(d)}E_2^\mu E_1^{(a-b)} K_2^{-(\mu+2\beta)}F_1^{(r-2e-d-b)}F_2^\nu F_1^{(\alpha(\mu+\nu+2\beta)+1-r-2c-a)}, \notag
 \end{align*}
and %
\begin{align*}
A_2 &=  \sum_{r=1,2\nmid r}^{\alpha(\mu+\nu+2\beta)+1}\sum_{c=0}^{\frac{\alpha}{2}(\mu+\nu+2\beta)+\frac{1-r}{2}} \sum_{e=0}^{\frac{r-1}{2}}
\sum_{a=0}^{\alpha+1-r-2c}\sum_{d=0}^{r-2e}\sum^{\min\{a,r-2e-d\}}_{b=0}
\notag\\
& q_1^{(a+c+d+e)(\alpha(\mu+\nu+2\beta)+2-r-2\la-2a-2c-2d-2e)+d+2\alpha(\mu+\beta)(d+e)} q_1^{\alpha(\mu+2\beta)(r+a-b-2e-d)} \notag \\
&   \cdot
\qbinom{2\alpha\mu +\alpha\nu+4\alpha\beta+3-2e-d-3a-2\la-4c-r}{b}_{q_1} \notag\\
&  \cdot
\qbinom{\frac{\alpha}{2}(\mu+\nu+2\beta)+\frac{1-r}{2}-c-a-\la}{c}_{q_1^2}
\qbinom{\frac{3\alpha\mu}{2} +\frac{\alpha\nu}{2}+2\alpha\beta+1-e-d-\la-2a-\frac{r-1}{2}-2c}{e}_{q_1^2}
\notag \\
&  \cdot  E_1^{(d)}E_2^\mu E_1^{(a-b)} K_2^{-(\mu+2\beta)}F_1^{(r-2e-d-b)}F_2^\nu F_1^{(\alpha(\mu+\nu+2\beta)+1-r-2c-a)}\onestar_{2\la-1}. \notag
\end{align*}

Let $l:=a+d-b$, $y:=r-2e-d-b$, $u:=c+e+b$, and $w:=\alpha(\mu+\nu+2\beta)+3-2\lambda-4u-2l-y+d$.
Using these new variables and $T(w,u,l,\mu,\beta)$ in \eqref{eq:T1jfix}, the above identity can be rewritten in the following form (analogous to \eqref{eq:evev3jfix} at the end of \S\ref{subsec:proof1})
\begin{align}
\label{eq:oddodd3jfix}
& \sum_{r=0}^{\alpha(\mu+\nu+2\beta)+1}
(-1)^r    B_{1,\odd}^{(r)}E_2^\mu K_2^{-(\mu+2\beta)}F_2^\nu B_{1,\odd}^{(\alpha(\mu+\nu+2\beta)+1-r)}\onestar_{2\la-1} \\ \notag
& =-\sum_{u=0}^{\frac{\alpha}{2}(\mu+\nu+2\beta)}\sum_{l=0}^{\alpha(\mu+\nu+2\beta)+1-2u} \sum_{d=0}^{l} \sum_{y=0}^{\alpha(\mu+\nu+2\beta)+1-l-2u}
\\
&\qquad q_1^{d+2\alpha\beta(\alpha\mu +l+y)+\alpha\mu (\alpha(\mu+\nu)+3-2\la-4u-l+2d)+(l+u)(u-2d-1)+l u}T(w,u,l,\mu,\beta)  \notag\\
&\qquad \cdot E_1^{(d)}E_2^\mu E_1^{(l-d)} K_2^{-(\mu+2\beta)}F_1^{(y)}F_2^\nu F_1^{(\alpha(\mu+\nu+2\beta)+1-l-y-2u)}\onestar_{2\la-1}. \notag
\end{align}
From now on, following the same proof for \eqref{eq:serre11F} as in \S\ref{subsec:ReductionGH}, we establish the identity \eqref{eq:serre11odd}; the details are omitted here.

\subsection{Proof of the identity \eqref{eq:serre11oddeven}}


Let $\alpha =-a_{12}$ as before. We shall use \eqref{t2mdot}--\eqref{t2m-1dot2} to rewrite the element
\begin{align}
\label{eq:Serre-idem2}
\sum_{r=0}^{\alpha(\mu+\nu+2\beta)+1} (-1)^r  B_{1,\odd}^{(r)}E_2^\mu K_2^{-(\mu+2\beta)}F_2^\nu  B_{1,\ev}^{(\alpha(\mu+\nu+2\beta)+1-r)}\onestar_{2\la} \in \Udot
\end{align}
for any $\la\in\Z$ in terms of monomials in $E_1, F_1, E_2, F_2, \tK^{-1}_2$.

Similar to \eqref{eq:evev}, we obtain the following formula:
\begin{align*}
& \sum_{r=0}^{\alpha(\mu+\nu+2\beta)+1}
 (-1)^r B_{1,\odd}^{(r)}E_2^\mu K_2^{-(\mu+2\beta)}F_2^\nu B_{1,\ev}^{(\alpha(\mu+\nu+2\beta)+1-r)}\onestar_{2\la} =C_1 -C_2
  \end{align*}
 where
\begin{align*}
C_1 & =\sum_{r=0,2\mid r}^{\alpha(\mu+\nu+2\beta)+1}\sum_{c=0}^{\frac{\alpha(\mu+\nu+2\beta)+1-r}{2}}\sum_{e=0}^{\frac{r}{2}} \sum_{a=0}^{\alpha(\mu+\nu+2\beta)+1-r-2c}\sum_{d=0}^{r-2e}\sum^{\min\{a,r-2e-d\}}_{b=0}
\\
&q_1^{(a+c+d+e)((\alpha(\mu+\nu+2\beta)+2-r-2\la-2a-2c-2d-2e)+2\alpha(\mu+\beta)(d+e)-a-2c}q_1^{\alpha(\mu+2\beta)(r+a-b-2e-d)}\notag\\
&\cdot
\qbinom{2\alpha\mu+\alpha\nu+4\alpha\beta+2-2e-d-3a-2\la-4c-r}{b}_{q_1}\notag\\
&  \cdot  \qbinom{\frac{\alpha(\mu+\nu+2\beta)+1-r}{2}-c-a-\la}{c}_{q_1^2}
\qbinom{\frac{3\alpha\mu+\alpha\nu+1}{2}+2\alpha\beta-e-d-\la-2a-\frac{r}{2}-2c}{e}_{q_1^2}
\notag \\
&  \cdot  E_1^{(d)}E_2^\mu E_1^{(a-b)} K_2^{-(\mu+2\beta)}F_1^{(r-2e-d-b)}F_2^\nu F_1^{(2\alpha(\mu+\nu+2\beta)+1-r-2c-a)},
  \end{align*}
and %
\begin{align*}
C_2&= \sum_{r=1,2\nmid r}^{\alpha(\mu+\nu+2\beta)+1}\sum_{c=0}^{\frac{\alpha(\mu+\nu+2\beta)-r}{2}}  \sum_{e=0}^{\frac{r-1}{2}}
\sum_{a=0}^{\alpha(\mu+\nu+2\beta)+1-r-2c} \sum_{d=0}^{r-2e}\sum^{\min\{a,r-2e-d\}}_{b=0}
\notag\\
&q_1^{(a+c+d+e)(\alpha(\mu+\nu+2\beta)+1-r-2\la-2a-2c-2d-2e)+d+2\alpha(\mu+\beta)(d+e)} q_1^{\alpha(\mu+2\beta)(r+a-b-2e-d)} \notag \\
&   \cdot
\qbinom{2\alpha\mu+\alpha\nu+4\alpha\beta+2-2e-d-3a-2\la-4c-r}{b}_{q_1} \notag \\
&   \cdot \qbinom{\frac{\alpha(\mu+\nu+2\beta)-r}{2}-c-a-\la}{c}_{q_1^2}
\qbinom{\frac{3\alpha\mu+\alpha\nu+1}{2}+2\alpha\beta-e-d-\la-2a-\frac{r-1}{2}-2c}{e}_{q_1^2}
  \notag\\
&\cdot E_1^{(d)}E_2^\mu K_2^{-(\mu+2\beta)}E_1^{(a-b)}F_1^{(r-2e-d-b)}F_2^\nu F_1^{(\alpha(\mu+\nu+2\beta)+1-r-2c-a)}\onestar_{2\la}. \notag
\end{align*}

Let $l:=a+d-b$, $y:=r-2e-d-b$, $u:=b+c+e$, and $w:=\alpha(\mu+\nu+2\beta)+2-2\lambda-4u-2l-y+d$.
Using these new variables and $T(w,u,l,\mu,\beta)$ in \eqref{eq:T1jfix}, the above identity can be rewritten in the following form (analogous to \eqref{eq:evev3jfix} at the end of \S\ref{subsec:proof1})
\begin{align}
\label{eq:oddeven3jfix}
& \sum_{r=0}^{\alpha(\mu+\nu+2\beta)+1}
(-1)^r    B_{1,\odd}^{(r)}E_2^\mu K_2^{-(\mu+2\beta)}F_2^\nu B_{1,\ev}^{(2\alpha(\mu+\nu+2\beta)+1-r)}\onestar_{2\la} \\ \notag
& =-\sum_{u=0}^{\alpha(\mu+\nu+2\beta)+1}\sum_{l=0}^{\alpha(\mu+\nu+2\beta)+1-2u} \sum_{d=0}^{l} \sum_{y=0}^{\alpha(\mu+\nu+2\beta)+1-l-2u}
\\
&\qquad q_1^{d+2\alpha\beta(\alpha\mu+l+y)+\alpha\mu(\alpha(\mu+\nu)+2-2\la-4u-l+2d)+(l+u)(u-2d-1)+l u}T(w,u,l,\mu,\beta)  \notag\\
&\qquad \cdot E_1^{(d)}E_2^\mu E_1^{(l-d)} K_2^{-(\mu+2\beta)}F_1^{(y)}F_2^\nu F_1^{(2\alpha(\mu+\nu+2\beta)+1-l-y-2u)}\onestar_{2\la}. \notag
\end{align}
From now on, following the same proof for \eqref{eq:serre11F} as in \S\ref{subsec:ReductionGH}, we establish the identity \eqref{eq:serre11oddeven}; the details are omitted here.

\subsection{Proof of the identity \eqref{eq:serre11evenodd} }

Let $\alpha =-a_{12}$ as before. We shall use \eqref{t2mdot}--\eqref{t2m-1dot2} to rewrite the element
\begin{align}
\sum_{r=0}^{\alpha(\mu+\nu+2\beta)+1} (-1)^r  B_{1,\ev}^{(r)}E_2^\mu K_2^{-(\mu+2\beta)}F_2^\nu  B_{1,\odd}^{(\alpha(\mu+\nu+2\beta)+1-r)}\onestar_{2\la-1} \in \Udot
\end{align}
for any $\la\in\Z$ in terms of monomials in $E_1, F_1, E_2, F_2, \tK^{-1}_2$.

Similar to \eqref{eq:evev}, we obtain the following formula:
\begin{align*}
& \sum_{r=0}^{\alpha(\mu+\nu+2\beta)+1}
 (-1)^r B_{1,\ev}^{(r)}E_2^\mu K_2^{-(\mu+2\beta)}F_2^\nu B_{1,\odd}^{(\alpha(\mu+\nu+2\beta)+1-r)}\onestar_{2\la-1} = D_1 -D_2
\end{align*}
 where
\begin{align*}
D_1
& =\sum_{r=0,2\mid r}^{\alpha(\mu+\nu+2\beta)+1}\sum_{c=0}^{\frac{\alpha(\mu+\nu+2\beta)+1-r}{2}}\sum_{e=0}^{\frac{r}{2}} \sum_{a=0}^{\alpha(\mu+\nu+2\beta)+1-r-2c}\sum_{d=0}^{r-2e}\sum^{\min\{a,r-2e-d\}}_{b=0}
\\
&q_1^{(a+c+d+e)(\alpha(\mu+\nu+2\beta)+2-r-2\la-2a-2c-2d-2e)+2\alpha(s+\beta)(d+e)+d}q_1^{\alpha(s+2\beta)(r+a-b-2e-d)}\notag\\
&\cdot
\qbinom{2\alpha\mu+\alpha\nu+4\alpha\beta+3-2e-d-3a-2\la-4c-r}{b}_{q_1}\notag\\
&  \cdot  \qbinom{\frac{\alpha(\mu+\nu+2\beta)+1-r}{2}-c-a-\la}{c}_{q_1^2}
\qbinom{\frac{3\alpha\mu+\alpha\nu+1}{2}+2\alpha\beta+1-e-d-\la-2a-\frac{r}{2}-2c}{e}_{q_1^2}
\notag\\
& \cdot E_1^{(d)}E_2^\mu E_1^{(a-b)} K_2^{-(\mu+2\beta)}F_1^{(r-2e-d-b)}F_2^\nu F_1^{(\alpha(\mu+\nu+2\beta)+1-r-2c-a)},
  \end{align*}
and %
\begin{align*}
D_2 &= \sum_{r=1,2\nmid r}^{\alpha(\mu+\nu+2\beta)+1}\sum_{c=0}^{\frac{\alpha(\mu+\nu+2\beta)-r}{2}}  \sum_{e=0}^{\frac{r-1}{2}}
\sum_{a=0}^{\alpha(\mu+\nu+2\beta)+1-r-2c} \sum_{d=0}^{r-2e}\sum^{\min\{a,r-2e-d\}}_{b=0}
\notag\\
&q_1^{(a+c+d+e)(\alpha(\mu+\nu+2\beta)+3-r-2\la-2a-2c-2d-2e)-a-2c+2\alpha(s+\beta)(d+e)} q_1^{\alpha(s+2\beta)(r+a-b-2e-d)} \notag \\
&   \cdot
\qbinom{2\alpha\mu+\alpha\nu+4\alpha\beta+3-2e-d-3a-2\la-4c-r}{b}_{q_1}\notag\\
&   \cdot \qbinom{\frac{\alpha(\mu+\nu+2\beta)-r}{2}-c-a-\la+1}{c}_{q_1^2}
\qbinom{\frac{3\alpha\mu+\alpha\nu+1}{2}+2\alpha\beta-\frac{r-1}{2}-e-d-\la-2a-2c}{e}_{q_1^2}
  \notag\\
&\cdot E_1^{(d)}E_2^\mu K_2^{-(\mu+2\beta)}E_1^{(a-b)}F_1^{(r-2e-d-b)}F_2^\nu F_1^{(\alpha(\mu+\nu+2\beta)+1-r-2c-a)}\onestar_{2\la-1}. \notag
\end{align*}

Let $l:=a+d-b$, $y:=r-2e-d-b$, $u:=b+c+e$, and $w:=\alpha(\mu+\nu+2\beta)+3-2\lambda-4u-2l-y+d$.
Using these new variables and $T(w,u,l,\mu,\beta)$ in \eqref{eq:T1jfix}, the above identity can be rewritten in the following form (analogous to \eqref{eq:evev3jfix} at the end of \S\ref{subsec:proof1})
\begin{align}
\label{eq:oddeven3jfix}
& \sum_{r=0}^{\alpha(\mu+\nu+2\beta)+1}
(-1)^r    B_{1,\ev}^{(r)}E_2^s K_2^{-(\mu+2\beta)}F_2^\nu B_{1,\odd}^{(2\alpha(\mu+\nu+2\beta)+1-r)}\onestar_{2\la-1} \\ \notag
& =\sum_{u=0}^{\alpha(\mu+\nu+2\beta)+1}\sum_{l=0}^{\alpha(\mu+\nu+2\beta)+1-2u} \sum_{d=0}^{l} \sum_{y=0}^{\alpha(\mu+\nu+2\beta)+1-l-2u}
\\
&\qquad q_1^{d+2\alpha\beta(\alpha\mu+l+y)+\alpha\mu(\alpha(\mu+\nu)+3-2\la-4u-l+2d)+(l+u)(u-2d-1)+l u} T(w,u,l,\mu,\beta) \notag\\
&\qquad \cdot E_1^{(d)}E_2^\mu E_1^{(l-d)} K_2^{-(\mu+2\beta)}F_1^{(y)}F_2^\nu F_1^{(2\alpha(\mu+\nu+2\beta)+1-l-y-2u)}\onestar_{2\la-1}. \notag
\end{align}
From now on, following the same proof for \eqref{eq:serre11F} as in \S\ref{subsec:ReductionGH}, we establish the identity \eqref{eq:serre11evenodd}; the details are omitted here.

\section{Some $q$-binomial combinatorial formulas}
  \label{App:B}

In this appendix, we establish two technical $q$-binomial combinatorial formulas, which are used in the proof of Serre-Lusztig relations in \S\ref{subsec:even}--\ref{subsec:odd}.

Recall the quantum binomial identity
\begin{align}
  \label{eq:qbi}
\qbinom{k+1}{u}_{q^2}=q^{-2eu}\qbinom{k}{u}_{q^2}+q^{e(2k-2u+2)}\qbinom{k}{u-1}_{q^2}.
\end{align}

\begin{lem} \label{lem:eqn1}
For any $-a,m,r,s,u\in \bbZ_{\geq 0}$ such that $r+s+2u=m+1$ and $m+na$ is even, we have
\begin{align}\label{eq:eqn1}
&q^{-e(2m+na)}[r]\qbinom{\frac{m+na}{2}}{u}_{q^2}
+
q^{-e(m+na-1)}[s]\qbinom{\frac{m+na-2}{2}}{u}_{q^2}
+[s+1]\qbinom{\frac{m+na-2}{2}}{u-1}_{q^2}\\ \notag
=& q^{-e(m+na+r-1)}[m+1]\qbinom{\frac{m+na}{2}}{u}_{q^2}-q^{-e(2m+na+1-r-2u)}[m+na-1]\qbinom{\frac{m+na-2}{2}}{u-1}_{q^2}.
\notag
\end{align}
\end{lem}

\begin{proof}
For $u=0$, the desired identity \eqref{eq:eqn1} can be verified directly.

Let $u\geq 1$. Using the quantum binomial identity \eqref{eq:qbi} we have
\begin{align*}
 \text{LHS}\, \eqref{eq:eqn1}
=&q^{-e(2m+na)}[r]\qbinom{\frac{m+na}{2}}{u}_{q^2}+
 q^{-e(m+na-1)}[s]\qbinom{\frac{m+na-2}{2}}{u}_{q^2}
+[s+1]\qbinom{\frac{m+na-2}{2}}{u-1}_{q^2}\\
=&
q^{-e(2m+na)}[r]\left(q^{-2eu}\qbinom{\frac{m+na-2}{2}}{u}_{q^2}
+q^{e(m+na-2u)}\qbinom{\frac{m+na-2}{2}}{u-1}_{q^2}\right)\\
&+q^{-e(m+na-1)}[s]\qbinom{\frac{m+na-2}{2}}{u}_{q^2}
+[s+1]\qbinom{\frac{m+na-2}{2}}{u-1}_{q^2}\\
=&\left(q^{-e(2m+na+2u)}[r]+q^{-e(m+na-1)}[s]  \right)\qbinom{\frac{m+na-2}{2}}{u}_{q^2}
\\
& +\left( q^{-e(m+2u)}[r]+[s+1] \right)\qbinom{\frac{m+na-2}{2}}{u-1}_{q^2},
\end{align*}
which can be rewritten, thanks to  $r+s+2u=m+1$,  as
\begin{align*}
=&\left(q^{-e(m+na-1+r+2u)}[m+1]-q^{-e(2m+na-r)}[2u]\right)\qbinom{\frac{m+na-2}{2}}{u}_{q^2}\\
&+\left( q^{-e(2u+r-1)}[m+1]-q^{-e(m+1-r)}[2u-1] \right)\qbinom{\frac{m+na-2}{2}}{u-1}_{q^2}\\
=&[m+1]\left(q^{-e(m+na-1+r+2u)}\qbinom{\frac{m+na-2}{2}}{u}_{q^2}+ q^{-e(2u+r-1)}\qbinom{\frac{m+na-2}{2}}{u-1}_{q^2} \right)\\
&-\left(q^{-e(2m+na-r)}[2u]\qbinom{\frac{m+na-2}{2}}{u}_{q^2}+ q^{-e(m+1-r)}[2u-1] \qbinom{\frac{m+na-2}{2}}{u-1}_{q^2}\right)\\
=&q^{-e(m+na-1+r)}[m+1]\left(q^{-2eu}\qbinom{\frac{m+na-2}{2}}{u}_{q^2}+ q^{-e(-m-na+2u)}\qbinom{\frac{m+na-2}{2}}{u-1}_{q^2}\right)\\
&-q^{-e(m+1-r)}\left(q^{-e(m+na-1)}[m+na-2u]+[2u-1]\right)\qbinom{\frac{m+na-2}{2}}{u-1}_{q^2}
=\text{RHS}\, \eqref{eq:eqn1}.
\end{align*}
The lemma is proved.
\end{proof}

\begin{lem}
\label{lem:eqn2}
For any $-a,m,r,s,u\in \bbZ_{\geq 0}$ such that $r+s+2u=m+1$ and $m+na$ is odd, we have
\begin{align}\label{eq:eqn2}
&q^{-e(2m+na-1)}[r]\qbinom{\frac{m+na-1}{2}}{u}_{q^2}
+q^{-e(2u+m+na-2)}[s]\qbinom{\frac{m+na-1}{2}}{u}_{q^2}\\ \notag
&\quad \quad \quad +q^{-e(3m+2na-1)}[r+1]\qbinom{\frac{m+na-1}{2}}{u-1}_{q^2}
+q^{-e(2u-2)}[s+1]\qbinom{\frac{m+na-1}{2}}{u-1}_{q^2}\\ \notag
 = & q^{-e(m+na+2u+r-2)}[m+1]\qbinom{\frac{m+na+1}{2}}{u}_{q^2}
 -q^{-e(2m+na-r-1)}[m+na-1]\qbinom{\frac{m+na-1}{2}}{u-1}_{q^2}.
 \notag
\end{align}

\end{lem}

\begin{proof}
For $u=0$, the desired identity \eqref{eq:eqn2} can be verified directly.

Let $u\geq 1$. Using the quantum binomial identity \eqref{eq:qbi} we have
\begin{align}\label{eqn:21}
& q^{-e(2m+na+2u-1)} [r] \left(q^{2eu}\qbinom{\frac{m+na-1}{2}}{u}_{q^2}\right)
+q^{-e(m+na-2)} [s] \left(q^{-2eu}\qbinom{\frac{m+na-1}{2}}{u}_{q^2}\right)
\\ \notag
 =& q^{-e(2m+na+2u-1)}  [r] \left(\qbinom{\frac{m+na+1}{2}}{u}_{q^2}-q^{-e(m+na-2u+1)}\qbinom{\frac{m+na-1}{2}}{u-1}_{q^2}\right)
 \\ \notag
& \quad
+q^{-e(m+na-2)} [s] \left(\qbinom{\frac{m+na+1}{2}}{u}_{q^2}-q^{-e(2u-m-na-1)}\qbinom{\frac{m+na-1}{2}}{u-1}_{q^2}\right)
\\ \notag
=& (q^{-e(2m+na+2u-1)}[r]+q^{-e(m+na-2)}[s])\qbinom{\frac{m+na+1}{2}}{u}_{q^2}
 \\
 &\notag
 \quad -(q^{-e(3m+2na)}[r]+q^{-e(2u-3)}[s])\qbinom{\frac{m+na-1}{2}}{u-1}_{q^2}
 \\ \notag
  =& q^{-e(m+na+2u+r-2)}[m+1]\qbinom{\frac{m+na+1}{2}}{u}_{q^2}
   -q^{-e(2m+na-r-1)}[2u]\qbinom{\frac{m+na+1}{2}}{u}_{q^2}
\\ \notag
&\quad -(q^{-e(3m+2na)}[r]+q^{-e(2u-3)}[s])\qbinom{\frac{m+na-1}{2}}{u-1}_{q^2}.
\end{align}

Recall $r+s+2u=m+1$. By a direct computation, we have
\begin{align}
 \label{eq:qq}
q^{-e(3m+2na-1)} & [r+1]
+q^{-e(2u-2)}[s+1]
-(q^{-e(3m+2na)}[r]+q^{-e(2u-3)}[s])\\
=&q^{-e(2m+na-r-1)}(q^{m+na}+q^{-m-na}).
\notag
\end{align}
Therefore, by \eqref{eqn:21}--\eqref{eq:qq} we have
\begin{align*}
\text{LHS}\, \eqref{eq:eqn2} =
&q^{-e(m+na+2u+r-2)}[m+1]\qbinom{\frac{m+na+1}{2}}{u}_{q^2}
-q^{-e(2m+na-r-1)}[2u]\qbinom{\frac{m+na+1}{2}}{u}_{q^2}
\\
& \quad +q^{-e(2m+na-r-1)}(q^{m+na}+q^{-m-na})
\qbinom{\frac{m+na-1}{2}}{u-1}_{q^2} \notag \\ \notag
=&q^{-e(m+na+2u+r-2)}[m+1]\qbinom{\frac{m+na+1}{2}}{u}_{q^2}
  \notag \\
  & \quad -q^{-e(2m+na-r-1)}[m+na+1]\qbinom{\frac{m+na-1}{2}}{u-1}_{q^2}
 \notag \\
  &\quad +q^{-e(2m+na-r-1)}(q^{m+na} +q^{-m-na}) \qbinom{\frac{m+na-1}{2}}{u-1}_{q^2}
\notag
= \text{RHS}\, \eqref{eq:eqn2}.
\end{align*}
This proves the lemma.
\end{proof}


\end{document}